\documentclass[12pt]{article}
\usepackage{defaults}

\tikzset{
    labl/.style={anchor=south, rotate=90, inner sep=.5mm}
}

\RequirePackage[style=alphabetic,doi=false,isbn=false,url=false,eprint=false]{biblatex}
\AtEveryBibitem{\clearlist{language}}
\begin{document}

\title{On the local-to-global principle for zero-cycles on self products of elliptic curves with CM}
\author{Michael Wills\footnote{Department of Mathematics, University of Virginia. 401 Kerchof Hall, 141 Cabell Dr., Charlottesville, VA 22904, USA. Email: \href{mailto:muf5cz@virginia.edu}{muf5cz@virginia.edu}}}
\date{}
\maketitle

\begin{abstract}
    For a smooth projective variety \(X\) defined over a global field \(K\), one can form a notion of Weak Approximation for the Chow group of zero-cycles of \(X\). There exists a Brauer-Manin obstruction to Weak Approximation here akin to that for rational points. However, unlike for rational points, it is conjectured that this obstruction is the only one; early versions of this conjecture date back to the 1980's \cite{Colliot-Thelene_Sansuc_1981}, \cite{Kato_Saito_86}. In this paper, we provide evidence for this when \(X\) is the self-product of an elliptic curve with complex multiplication. For some varieties of this form, we construct infinitely many extensions \(L/K\) for which the base change \(X\times_K \Spec L\) satisfies a local-to-global principle for a fixed prime \(p\). We do this via explicitly constructing global zero-cycles, and our results have applications over all but two of the complex multiplication fields.
\end{abstract}

\section{Introduction}
Let \(K\) be a number field with set of places \(\Omega\), and for each \(v\in \Omega\) let \(\iota_v:K\hookrightarrow K_v\) denote the embedding of \(K\) into the local field corresponding to \(v\).
Let \(X\) be a smooth projective geometrically connected variety over \(K\), and consider its set \(X(K)\) of \(K\)-rational points. One approach to understanding these points is consider the diagonal embedding of \(X(K)\) into the set of \textit{adelic points} of \(X\) defined by \(X(\mathbb{A}_K) = \prod_{v\in \Omega} X_v(K_v)\), where for each \(v\), \(X_v = X \times_K \Spec K_v\) denotes the base change. In particular, one can ask:
\begin{enumerate}
    \item does \(X\) satisfy the \textit{Hasse Principle}, i.e.\ does the existence of an adelic point for \(X\) imply existence of a \(K\)-rational point for \(X\), and if so
    \item does \(X\) satisfy \textit{Weak Approximation}, i.e.\ is \(X(K)\) dense in \(X(\mathbb{A}_K)\) (in the appropriate topological sense)?
\end{enumerate}
One obstruction to both of these arises from the cohomological Brauer group \(\Br(X) = H^2(X, \mathbb{G}_m)\); more precisely, in \cite{Manin_1971} a pairing of sets \(X(\mathbb{A}_K) \times \Br(X) \to \mathbb{Q}/\mathbb{Z}\) is constructed under which points in \(X(K)\) pair trivially with all elements of \(\Br(X)\). The left (pointed set) kernel of this pairing then defines a closed intermediate set \[X(K) \subseteq X(\mathbb{A}_K)^{\Br} \subseteq X(\mathbb{A}_K),\] and we say that the Brauer-Manin obstruction \textit{explains the failure of the Hasse Principle} (or \textit{Weak Approximation}) if \(X(\mathbb{A}_K) \neq \emptyset\) but \(X(\mathbb{A}_K)^{\Br} = \emptyset\) (or if \(X(\mathbb{A}_K)^{\Br}\neq X(\mathbb{A}_K)\) respectively). However, this does not capture all ways that either the Hasse Principle or Weak Approximation may fail; see \cite{Skorobogatov_1999} for the first explicit example, and \cite[Chapter 8]{Poonen_2017} for an overview of further refinements to this question.

Similar questions arise for the \textit{Chow group of zero-cycles} \(\CH_0(X)\) of \(X\), a ``linearization'' of the closed points of \(X\) foundational to intersection theory \cite{Fulton_1998}. Namely, one can construct an adelic analogue to \(\CH_0(X)\), which for \(K\) totally imaginary is given by
\[
    \CH_{0,\mathbb{A}}(X) = \prod_{v\in\Omega_f} \CH_0(X_{K_v})
\]
where \(\Omega_f\subset \Omega\) denotes the finite places (see \autoref{ssec:filtration_on_CH} for the general definition), together with a diagonal map \(\Delta:\CH_0(X) \to \CH_{0,\mathbb{A}}(X)\).
The Brauer-Manin pairing on rational points then extends to an abelian group pairing \(\CH_{0,\mathbb{A}}(X) \times \Br(X) \to \mathbb{Q}/\mathbb{Z}\), and the image of \(\Delta\) is contained in the left kernel of this pairing \cite[p. 57]{Colliot-Thelene_1995}. Since \(\Br(X)\) is torsion, this extends to a pairing \(\widehat{\CH_{0,\mathbb{A}}(X)} \times \Br(X) \to \mathbb{Q}/\mathbb{Z}\) (where for any abelian group \(A\), we denote by \(A/n\) the quotient \(A/nA\), and by \(\widehat{A}\) the profinite completion \(\limi_n A/n\)), giving rise to a complex
\begin{align}\label{eqn:global_CH_complex}
    \widehat{\CH_0(X)} \overset{\Delta}{\rightarrow} \widehat{\CH_{0,\mathbb{A}}(X)} \overset{\varepsilon}{\rightarrow} \Hom(\Br(X), \mathbb{Q}/\mathbb{Z}).
\end{align}

We then have the following Conjecture due to Colliot-Th\'el\`ene:

\begin{conj*}[E]\label{conj:E}
For a smooth projective geometrically connected variety \(X\) defined over a number field, the complex (\ref{eqn:global_CH_complex}) is exact.
\end{conj*}

As discussed in \cite{Wittenberg_2012}, this conjecture implies both that the Brauer-Manin obstruction is the only obstruction to the Hasse Principle for zero-cycles of degree 1, as well as a version of Weak Approximation for zero-cycles (now in a group-theoretic rather than topological sense). Conjecture (E) goes back to conjectures in \cite{Colliot-Thelene_Sansuc_1981} on geometrically rational surfaces and in \cite{Kato_Saito_86} on higher class field theory; see \cite{Colliot-Thelene_1995}, \cite{Wittenberg_2012}, and \cite{Colliot_Thelene_Skorobogatov_2021} for a fuller history. The form given above is due to \cite{van_Hamel_03} and \cite{Wittenberg_2012}.

Little is known about Conjecture (E) in general, especially unconditionally. For \(X = \Spec K\), Conjecture (E) follows from the fundamental short exact sequence of global class field theory. For \(X\) a curve, Conjecture (E) was proved in \cite{Colliot-Thelene_1999}, conditional on finiteness of the Shafarevich-Tate group of the Jacobian of \(X\). In higher dimensions, there are conditional results given in \cite{Liang_2013} (building upon work in \cite{Colliot-Thelene_Sansuc_Swinnerton-Dyer_1987}, \cite{Colliot-Thelene_Sansuc_Swinnerton-Dyer_1987a}), which show Conjecture (E) for rationally connected varieties under the hypothesis that the Brauer-Manin obstruction is the only one for Weak Approximation of rational points after any finite base change. A similarly conditional result on the Hasse Principle for zero-cycles of odd degree on products of Kummer varieties may be found in \cite{Balestrieri_Newton_2021}. For Weak Approximation on K3 surfaces, a weaker ``fixed \(n\)'' version of Conjecture (E) is shown in \cite{Ieronymou_2021}, again conditional on a deep understanding of Weak Approximation for rational points. There are no known abelian varieties of dimension at least 2 for which Conjecture (E) holds in its entirety, despite knowing that Weak Approximation for rational points holds for certain classes of such varieties (assuming finiteness of the Tate-Shafarevich group) \cite{Wang_1996}. Additionally, compatibility of Conjecture (E) is shown in \cite{Harpaz_Wittenberg_2016} with certain fibrations over \(\mathbb{P}^1\), and in \cite{Liang_2023} with products of at most one nice curve and varieties satisfying a certain condition on their Brauer group, which includes geometrically rationally connected varieties and K3 surfaces.

The goal of this paper is to provide unconditional evidence of Conjecture (E) for infinite families of varieties \(X = (E\times_K E)_L\), where \(E\) is an elliptic curve defined over an imaginary quadratic field \(K\) having complex multiplication by the full ring of integers \(\mathcal{O}_K\), and \(L/K\) is a finite extension. We will do so via explicit computations with zero-cycles rather than using the arithmetic of rational points that powers the conditional results discussed above. 

To state the nature of this evidence, recall that \(\CH_0(X)\) admits a filtration \[\CH_0(X) \supseteq F^1(X) \supseteq F^2(X),\] where \(F^1(X)\) denotes the zero-cycles of degree zero and \(F^2(X)\) is the Albanese kernel, also known as the Abel-Jacobi kernel. This filtration and its adelic analogue are compatible with the maps in the complex (\ref{eqn:global_CH_complex}), so a necessary aspect of proving Conjecture (E) is showing exactness of
\begin{align}\label{eqn:global_F2_complex}
    \widehat{F^2(X)} \overset{\Delta}{\rightarrow} \widehat{F^2_{\mathbb{A}}(X)} \overset{\varepsilon}{\rightarrow} \Hom(\Br(X), \mathbb{Q}/\mathbb{Z}).
\end{align}
In fact, by \cite[Proposition 5.6]{Gazaki_Hiranouchi_2021}, exactness of (\ref{eqn:global_F2_complex}) is equivalent to Conjecture (E) if one assumes that the Tate-Shafarevich group of \(X\) contains no non-zero divisible element.

Note next that by the Chinese Remainder Theorem, exactness of (\ref{eqn:global_F2_complex}) is equivalent to exactness of
\begin{align}\label{eqn:p_primary_F2_complex}
    \limi_n F^2(X)/p^n \overset{\Delta}{\rightarrow} \limi_n F^2_{\mathbb{A}}(X)/p^n \overset{\varepsilon}{\rightarrow} \Hom(\Br(X), \mathbb{Q}/\mathbb{Z})
\end{align}
for each prime \(p\). As it turns out, for $X=(E\times E)_L$ and \(p\geq 5\), only the places of \(L\) lying above \(p\) will contribute to the middle term; see \autoref{lem:places_away_from_p_are_trivial}. This is very useful, as at present our understanding of the local \(F^2\) groups at places of supersingular reduction for \(X\) lags behind that at places of ordinary reduction; see \cite[Section 1.2]{Gazaki_Hiranouchi_2021} for some discussion of why this is the case. Since \(E\) has complex multiplication, the behavior of \(E_L\) at places lying above a rational prime \(p\) is entirely determined by how \(p\) behaves in the extension \(K/\mathbb{Q}\), with \(p\) inert implying that the reduction is supersingular and \(p\) splitting implying that it is ordinary \cite[Theorem 13.12]{Lang_1987}. This motivates the central study of this paper, which is to find specific examples of elliptic curves \(E/K\), finite extensions \(L/K\), and primes \(p\) splitting in \(K/\mathbb{Q}\) for which we can show exactness of (\ref{eqn:p_primary_F2_complex}) for \(X = (E\times E)_L\). We are specifically interested in examples where the middle term is non-vanishing, so that exactness is not trivially satisfied. This brings us to our first result:

\begin{theorem}[cf. Theorems \ref{thm:no_inertia_nontriviality} and \ref{thm:D_equals_1_nontriviality_criterion}]\label{thm:nontriviality_proportion_summary}
Let \(K\) be an imaginary quadratic field of class number 1, let \(p\) be a prime which splits in \(K/\mathbb{Q}\), and let \(E\) be an elliptic curve over \(K\) with complex multiplication by $\mathcal{O}_K$. Assume that $p \mid |\bar{E}(\mathbb{F}_p)|$, where $\bar{E}$ denotes the reduction of $E$ at some place of $K$ above $p$. Let $X = E\times_K E$.

There exist infinitely many extensions $L/K$ for which we may construct $z\in F^2(X_L)$ such that $\Delta(z)\neq 0$, where
\[
    \limi_n F^2(X_L)/p^n \overset{\Delta}{\longrightarrow} \prod_{v\mid p} \limi_n F^2(X_{L_v})/p^n.
\]
\end{theorem}

See \autoref{lem:all_cases_for_nontriv_via_vals_propn} for a list of situations in which this result may be applied; for seven of the nine CM types, one may find \(E\) and \(p\) as desired, and for six of these it is reasonable to expect that this occurs for infinitely many primes \(p\). The field extensions $L/K$ and zero-cycles $z$ referred to in \autoref{thm:nontriviality_proportion_summary} are explicitly characterized, and the algorithm for determining non-triviality of $\Delta(z)$ for the type of zero-cycle constructed is implemented in an accompanying Sage notebook\footnote{Available at \href{https://github.com/mwills758/locally_non-trivial_cycles/}{https://github.com/mwills758/locally\_non-trivial\_cycles/}}.

Previously, the only result of this form was found in \cite{Gazaki_Koutsianas_2024}. There, it is assumed that $E$ is already defined over $\mathbb{Q}$, and a conditional procedure is established for generating $z\in F^2(X_{L_0})$ with $\Delta(z) \neq 0$ for one extension $L_0/K$; this $L_0$ is the intersection of all $L$ appearing in our \autoref{thm:nontriviality_proportion_summary} for a given $E$ and $p$. The authors of this paper applied this result to several thousand elliptic curves with potential complex multiplication by the full ring of integers of \(\mathbb{Q}(\sqrt{-7})\) and with \(p=7\), computing that approximately 86.68\% of the curves of rank 1 admitted such a zero-cycle. Our \autoref{thm:nontriviality_proportion_summary} generalizes this procedure and gives an explanation for the proportion recorded in \cite{Gazaki_Koutsianas_2024}, demonstrating a precisely $\frac{p-1}{p}$ chance of $\Delta(z)$ non-triviality for any prime $p$ and CM type.

Fix $X = E\times_K E$ as before. For many choices of $L/K$ as in \autoref{thm:nontriviality_proportion_summary}, one can show that the map
\[
    \limi_n F^2_{\mathbb{A}}(X_L)/p^n \overset{\varepsilon}{\rightarrow} \Hom(\Br(X_L), \mathbb{Q}/\mathbb{Z})
\]
appearing in the complex (\ref{eqn:p_primary_F2_complex}) has trivial image, so proving exactness of (\ref{eqn:p_primary_F2_complex}) is equivalent to showing surjectivity of $\Delta$; see \cite[Theorem 4.2, Claim 3]{Gazaki_Koutsianas_2024}. The structure of the middle term for these $L$ is determined in \autoref{cor:size_of_middle_term} to be an $m$-dimensional $\mathbb{F}_p$-vector space for some $m$ depending on the splitting behavior of $p$ in $L/\mathbb{Q}$. While \autoref{thm:nontriviality_proportion_summary} allows us to construct extensions $L/K$ in which $\Im(\Delta)$ has positive $\mathbb{F}_p$-dimension, in general the value $m$ grows much too quickly for the zero-cycles we study to suffice. However, in some cases we are able to obtain surjectivity, and therefore exactness of (\ref{eqn:p_primary_F2_complex}):

\begin{theorem}[cf. \autoref{thm:main_theorem_detailed}]\label{thm:main_theorem_summary}
Let \(K\) be an imaginary quadratic field of class number 1, let \(p\) be a prime which splits in \(K/\mathbb{Q}\), and let \(E\) be an elliptic curve over \(K\) with complex multiplication by $\mathcal{O}_K$. Assume that $p \mid |\bar{E}(\mathbb{F}_p)|$, where $\bar{E}$ denotes the reduction of $E$ at some place of $K$ above $p$. Let $X = E\times_K E$.

Suppose that over the $L_0$ defined in the preceeding paragraph we may already construct a zero-cycle $z$ of the type produced by \autoref{thm:nontriviality_proportion_summary} such that $\Delta(z) \neq 0$. Then there exist infinitely many extensions $L/K$ for which $\Delta$ is surjective; that is, for which
\[
    \limi_n F^2(X_L)/p^n \overset{\Delta}{\rightarrow} \limi_n F^2_{\mathbb{A}}(X_L)/p^n \overset{\varepsilon}{\rightarrow} \Hom(\Br(X_L), \mathbb{Q}/\mathbb{Z})
\]
is exact.
\end{theorem}

All of the $L$ constructed in \autoref{thm:main_theorem_summary} are given as degree 2 extensions of $L_0$, and the middle term of the exact sequence is an $\mathbb{F}_p$-vector space of dimension $m=2$.

Examples of such curves are not particularly rare; among curves defined just over \(\mathbb{Q}\), one expects that about half have positive rank, and heuristically curves of positive rank meet the additional criterion of \autoref{thm:main_theorem_summary} with proportion \(\frac{p-1}{p}\). See \autoref{exm:CM_by_D_equals_43} for an application of both of these Theorems to a specific curve $E/\mathbb{Q}$ with complex multiplication by the ring of integers of $K = \mathbb{Q}(\sqrt{-43})$.

The structure of this paper is as follows. In \autoref{sec:reduction_to_K2}, we review a reduction of Conjecture (E) available for products of curves with a \(K\)-rational point. In \autoref{sec:local_structure}, we describe the structure of the adelic Albanese kernel \(\limi_n F^2_{\mathbb{A}}(E\times E)/p^n\). In \autoref{sec:non-trivial_symbols}, we apply these structural results to give explicit conditions for local non-triviality of global symbols. Finally, in \autoref{sec:linear_independence} we show how these results are used to give infinite families of full local-to-global principles for zero-cycles lying in the adelic Albanese kernel.

\subsection{Notation}
Throughout, we will use the following notation:
\begin{itemize}
    \item For \(A\) an abelian group and \(n \in \mathbb{N}\), we will let \(A/n\) denote \(A/nA\), and \(\widehat{A} = \limi_n A/n\). For \(p\) prime, we let \(A\{p\}\) denote the \(p\)-primary part of \(A\).
    \item For a number field \(F\), let \(\mathcal{O}_F\) denote its ring of integers, and let \(\Omega_f\) and \(\Omega_\infty\) denote the finite and infinite places of \(F\) respectively. For \(v \in \Omega_f \cup \Omega_\infty\), let \(\iota_v:F \hookrightarrow F_v\) the localization. For a variety \(X\) defined over \(F\), we denote by \(X_v\) the base change \(X_{F_v}\), and for \(P \in X(F)\) we let \(P_v\) the corresponding point in \(X_v(F_v)\).
    \item For a local field \(k\), let \(\mathcal{O}_k\) denote its ring of integers, \(\mathfrak{m}_k\) its maximal ideal, and \(\mathbb{F}_k = \mathcal{O}_k / \mathfrak{m}_k\) the residue field. If \(k = F_v\) for \(F\) a number field and \(v\in \Omega_f\), we will denote these objects by \(\mathcal{O}_v\), \(\mathfrak{m}_v\), and \(\mathbb{F}_v\) respectively.
    \item All tensor products are over \(\mathbb{Z}\) unless otherwise noted.    
\end{itemize}

\subsection{Acknowledgements}
I would like to thank Professors Jean-Louis Colliot-Th\'el\`ene and Toshiro Hiranouchi for constructive and detailed feedback throughout the paper; Professors David Harari, Rachel Newton, Alexei Skorobogatov, and Olivier Wittenberg for their time and helpful conversations; and Professor Valia Gazaki, my thesis advisor, without whose patient guidance this paper could not have happened.

\newpage

\section{Reduction to Somekawa \texorpdfstring{$K$}{K}-groups}\label{sec:reduction_to_K2}

In this section, we discuss how to use results in \cite{Raskind_Spiess_2000} to reduce questions about zero-cycles to questions about Somekawa \(K\)-groups, which are more amenable to computation.

\subsection{Mackey functors and Somekawa \texorpdfstring{\(K\)}{K}-groups}

Let \(K\) be a perfect field.

\begin{definition}\label{def:Mackey_functor}
A \textit{Mackey functor} \(F\) over \(K\) is a contravariant functor \(F\) from the category of \'etale \(K\)-schemes \(\Sch_{\et/K}\) to abelian groups, together with a covariant association taking finite morphisms \(f:A\to B\) in \(\Sch_{\et/K}\) to abelian group homomorphisms \(f_*: F(A) \to F(B)\) such that
\begin{enumerate}
\item \(F(A\sqcup B) = F(A) \oplus F(B)\), and
\item given a pullback diagram of \'etale schemes
\[\begin{tikzcd}
        A & {B_1} \\
        {B_2} & C
        \arrow["{f_1}", from=1-1, to=1-2]
        \arrow["{f_2}"', from=1-1, to=2-1]
        \arrow["{g_1}", from=1-2, to=2-2]
        \arrow["{g_2}"', from=2-1, to=2-2]
\end{tikzcd}\]
we have that
\[\begin{tikzcd}
        {F(A)} & {F(B_1)} \\
        {F(B_2)} & {F(C)}
        \arrow["{f_2^*}"', from=1-1, to=2-1]
        \arrow["{g_1^*}", from=1-2, to=2-2]
        \arrow["{(f_1)_*}", from=1-1, to=1-2]
        \arrow["{(g_2)_*}"', from=2-1, to=2-2]
\end{tikzcd}\]
commutes, where \(f^*\) denotes \(F(f)\).
\end{enumerate}
\end{definition}

If \(X = \Spec L\) for some \(L/K\) finite, we write \(F(L) = F(\Spec L)\), and if \(f:\Spec M \to \Spec L\) is induced by a finite field extension \(M/L\) we denote by \(\res_{M/L}\) and \(\cor_{M/L}\) the maps \(f^*\) and \(f_*\) respectively. Note that since \'etale schemes over \(K\) are all of the form \(\bigsqcup_{i=1}^n \Spec L_i\) for \(L_1,\ldots, L_n\) finite (and thus separable) extensions of \(K\), the first condition above implies that a Mackey functor \(F\) is completely determined by its values \(F(L)\) for all \(L/K\) finite. The two main examples of Mackey functors over \(K\) we will use are the following:
\begin{itemize}
    \item For a semi-abelian variety \(A/K\) (i.e.\ an extension of an abelian variety by an algebraic torus; note that this class contains both abelian varieties and \(\mathbb{G}_m\)), the association \(L/K\mapsto A(L)\) defines a Mackey functor which we also denote by \(A\), where the restriction and corestriction maps are given by the usual inclusions and norm maps on semi-abelian varieties.
    \item For a \(\Gal(\bar{K}/K)\)-module \(B\) and \(i\geq 0\), the association \(L/K \mapsto H^i(L, B)\) defines a Mackey functor denoted \(H^i(-,B)\), where the restriction and corestriction maps are those from group cohomology.
\end{itemize}

A \textit{morphism of Mackey functors} is a natural transformation which preserves the covariant structure. Mackey functors over \(K\) then form a category, and this category is abelian with kernels, cokernels, and products given coordinate-wise. If we have an injective morphism of Mackey functors \(F'\hookrightarrow F\), we say that \(F'\) is a \textit{sub-Mackey functor} of \(F\), and this is equivalent to having \(F'(L) \hookrightarrow F(L)\) for all \(L/K\) finite. Every Mackey functor admits a multiplication-by-\(n\) endomorphism for all \(n \in \mathbb{Z}\). We define the \textit{Mackey product} of Mackey functors \(F_1, \hdots, F_n\) to be the Mackey functor \(F_1\mprod \cdots \mprod F_n\) which has, for \(L/K\) finite,
\begin{align*}
    (F_1\mprod \cdots \mprod F_n)(L) = \left(\bigoplus_{M/L\ \text{finite}} F_1(M) \otimes \cdots \otimes F_n(M)\right)/R,
\end{align*}
where \(R\) is generated by the \textit{projection formula relations}
\begin{align*}
    \paren{x_1\otimes \cdots\otimes \cor_{M'/M}(y_i) \otimes \cdots \otimes x_n} - \paren{\res_{M'/M}(x_1)\otimes \cdots\otimes y_i \otimes \cdots \otimes \res_{M'/M}(x_n)}
\end{align*}
for all towers of finite extensions \(M'/M/L\), all \(i = 1,\ldots, m\), all \(y_i\in F_i(M')\), and all \({x_j\in F_j(M)}\) for \(j\neq i\). Elements of \((F_1\mprod \cdots \mprod F_n)(L)\) for some \(L/K\) are called \textit{symbols}, and the symbol represented by \(x_1\otimes \cdots \otimes x_n \in F_1(M) \otimes \cdots \otimes F_n(M)\) is denoted by \(\{x_1,\hdots, x_n\}_{M/L}\). The Mackey product forms a tensor product on the category of Mackey functors over \(K\) \cite[Appendix A]{Kahn_Yamazaki_2013}.

Now, let \(A_1, \hdots, A_n\) be semi-abelian varieties over a perfect field \(k\).

\begin{definition}
    \cite{Somekawa_1990} The \textit{Somekawa K-group} of \(A_1, \hdots, A_n\) is defined by
    \[
        K(k; A_1, \hdots, A_n) = (A_1\mprod \cdots \mprod A_n)(k)/R',
    \]
    where \(R'\) is generated by the \textit{Weil reciprocity relations} defined in \cite[Definition 1.2]{Somekawa_1990} (see \cite[Definition 2.1.1]{Raskind_Spiess_2000} for a correction to this definition). If \(A_i = A\) for each \(i=1,\hdots, n\), we write \(K_n(k;A) = K(k; A, \hdots, A)\).
\end{definition}

By \cite[Proposition 1.5]{Somekawa_1990}, for every integer \(m\) coprime to \(\chr k\) there exists a group homomorphism
\[
    s_m: K(k; A_1, \hdots, A_n)/m \to H^2(k, A_1[m] \otimes \cdots \otimes A_n[m])
\]
called the \textit{generalized Galois symbol}, which is uniquely characterized by, for all \(\ell/k\) finite, fitting into the commutative diagram 
\[\begin{tikzcd}
	{\bigotimes_{i=1}^n A_i(\ell)/m} && {K(k;A_1,\hdots, A_n)/m} \\
	{\bigotimes_{i=1}^n H^1(\ell,A_i[m])} && {H^m(k, A_1[m]\otimes \cdots \otimes A_n[m])}
	\arrow[from=1-1, to=1-3]
	\arrow["{\bigotimes \delta_i}"', from=1-1, to=2-1]
	\arrow["{s_m}", from=1-3, to=2-3]
	\arrow["{\cor_{\ell/k}\circ \smile}"', from=2-1, to=2-3]
\end{tikzcd}\]
where each \(\delta_i\) is the connecting homomorphism induced by the short exact sequence of \(\Gal(\bar{k}/\ell)\)-modules 
\[
0 \to A_i[m] \to A_i(\bar{k}) \overset{m}{\to} A_i(\bar{k}) \to 0
\]
and \(\smile\) denotes the usual cup product on group cohomology.

\subsection{The filtration on \texorpdfstring{\(\CH_0(X)\)}{CH\_0(X)}}\label{ssec:filtration_on_CH}

Let \(X\) be a smooth projective variety over any field \(k\). The Chow group of zero-cycles of \(X\) admits a filtration
\[
    \CH_0(X) \supseteq F^1(X) \supseteq F^2(X),
\]
defined as follows. The group \(F^1(X)\) (often denoted \(A_0(X)\) in the literature) is the kernel of the degree map \(\CH_0(X) \to \mathbb{Z}\) defined on classes of points by \([P] \mapsto [k(P):k]\). If \(X(k) \neq \emptyset\), we can fix a basepoint \(P_0 \in X(k)\) and generate \(F^1(X)\) by symbols of the form \([P] - [P_0]\). 

Now, let \(P_0 \in X(k)\) again be a basepoint, and let \(\Alb_X\) denote the corresponding Albanese variety of \(X\), which together with the Albanese map \(\alb_X:X \to \Alb_X\) is characterized by the property that any other map from \(X\) to an abelian variety \(A\) mapping \(P_0\) to the zero point of \(A\) factors through \(\alb_X\). The Albanese map induces a homomorphism \(\alb_X: F^1(X) \to \Alb_X(k)\), and we denote the kernel of this map by \(F^2(X)\).

Now suppose that \(X= C_1 \times C_2\) is a product of two smooth projective geometrically connected curves with \(C_i(k) \neq \emptyset\). In this case, the Albanese variety of \(X\) with basepoint \((P_0, Q_0)\) is simply the product of the Jacobians \(J_1\) and \(J_2\) correpsonding to the basepoints \(P_0\) and \(Q_0\) respectively. Since the curves involved have points, we can apply the tools developed in \cite{Raskind_Spiess_2000}, in particular Corollary 2.4.1, to get that
\begin{align*}
    \CH_0(X) \cong \mathbb{Z} \oplus \Big(K(k;J_1) \oplus K(k; J_2) \Big) \oplus K(k; J_1, J_2).
\end{align*}
The terms in this decomposition correspond to the filtration of \(\CH_0(X)\) given above. Note that \(K(k;J_i) = J_i(k)\) by definition. More importantly for us, for \(X\) of this form we have an explicit isomorphism \(K(k;J_1, J_2) \cong F^2(X)\) via
\begin{align*}
    \{P, Q\}_{\ell/k} \mapsto (\rho_{\ell/k})_* \big( [P, Q] - [P_0, Q] - [P, Q_0] + [P_0, Q_0] \big)
\end{align*}
where \(\rho_{\ell/k}:X_\ell \to X\) denotes the base change; see the discussion following \cite[Theorem 2.16]{Gazaki_Hiranouchi_2021} for details.

Now, suppose that \(X\) is defined over a number field \(K\).

\begin{definition}
    The \textit{adelic Chow group of zero-cycles} of \(X\) is defined to be the product
    \[
        \CH_{0,\mathbb{A}}(X) = \prod_{v \in \Omega_f} \CH_0(X_v) \times \prod_{v\in \Omega_{\infty}} \bar{\CH_0}(X_v),
    \]
    where for \(v\in \Omega_\infty\) we set \(\bar{\CH_0}(X_v) = \CH_0(X_v) / (\rho_v)_*(\CH_0(X_{\bar{K_v}}))\), where \(\rho_v:X_{\bar{K_v}} \to X_v\) is the base change.

    We may analogously obtain a filtration $\CH_{0,\mathbb{A}}(X) \supseteq F^1_{\mathbb{A}}(X) \supseteq F^2_{\mathbb{A}}(X)$ by setting
    \[
        F^i_{\mathbb{A}}(X) = \prod_{v \in \Omega_f} F^i(X_v) \times \prod_{v\in \Omega_{\infty}} \bar{F^i}(X_v)
    \]
    for $i=1,2$, where for $v\in \Omega_{\infty}$ we denote by $\bar{F^i}(X_v)$ the image of \(F^i(X_v)\) in \(\bar{\CH_0}(X_v)\).
\end{definition}
Note that \(\bar{\CH_0}(X_v)\) is zero for \(v\) complex, and by \cite[Th\'eor\`eme 1.3(c)]{Colliot-Thelene_1995} we see that the real places contribute only finitely many copies of \(\mathbb{Z}/2\) to the subgroups \(F^i(X_v)\) for \(i=1,2\).

Now let us focus on the case that \(X = E\times E\) is a self-product of elliptic curves with complex multiplication defined over \(K\). To analyze the last term of our complex, recall that the Brauer group of \(X\) admits a filtration \(\Br(X) \supseteq \Br_1(X) \supseteq \Br_0(X)\) coming from the Hochschild-Serre spectral sequence
\begin{align*}
    H^i(k, H^j(X_{\bar{k}}, \mathbb{G}_m)) \Rightarrow H^{i+j}(X, \mathbb{G}_m),
\end{align*}
where \(\Br_1(X)\) denotes the \textit{algebraic Brauer group} \(\ker(\Br(X) \to \Br(\bar{X}))\) and \(\Br_0(X)\) is given by \(\Im(\Br(k) \to \Br(X))\). As proved in \cite[Section 5.2]{Gazaki_Hiranouchi_2021} and \cite[Corollary 2.17]{Gazaki_Koutsianas_2024}, the map \(\varepsilon:\widehat{\CH_{0,\mathbb{A}}(X)} \to \Hom(\Br(X), \mathbb{Q}/\mathbb{Z})\) restricts to a map \(\widehat{F^2_{\mathbb{A}}(X)} \to \Hom(\Br(X)/\Br_1(X), \mathbb{Q}/\mathbb{Z})\). Working at a single prime \(p\), we obtain a further restriction
\begin{align*}
    \varepsilon: \limi_n F^2_{\mathbb{A}}(X)/p^n \to \Hom\left(\frac{\Br(X)\{p\}}{\Br_1(X)\{p\}}, \mathbb{Q}/\mathbb{Z}\right).
\end{align*}

When \(p\geq 3\), the potential 2-torsion in \(F^2_{\mathbb{A}}(X)\) arising from the real places of \(K\) is trivial modulo \(p\), so we summarize the above by rewriting the complex (\ref{eqn:p_primary_F2_complex}) as
\begin{align}\label{eqn:p_primary_K2_complex}
    \limi_n K_2(K;E)/p^n \overset{\Delta}{\rightarrow} \prod_{v \in \Omega_f}\limi_n K_2(K_v;E_v)/p^n \overset{\varepsilon}{\rightarrow} \Hom\left(\frac{\Br(X)\{p\}}{\Br_1(X)\{p\}}, \mathbb{Q}/\mathbb{Z}\right).
\end{align}

\begin{lemma}\label{lem:places_away_from_p_are_trivial}
Suppose that \(p\geq 5\), and let \(v\) a place of \(K\) not lying over \(p\). Then \[\limi_n F^2(X_v) = \limi_n K_2(K_v;E_v)/p^n = 0.\]
\end{lemma}
\begin{pf}
(Cf. \cite[Proof of Theorem 4.2, Claim 2]{Gazaki_Koutsianas_2024})    
If \(v\) is a place of good reduction for \(E\), then by \cite[Corollary 3.5.1(b)]{Raskind_Spiess_2000} we have that \(K_2(K_v;E_v)\) is \(p\)-divisible, and so the corresponding inverse limit vanishes. 

Now suppose that \(v\) is a place of bad reduction for \(E\). Since \(E\) has complex multiplication, we know that \(E_v\) attains good reduction after an extension \(k/K_v\) of degree dividing 6 \cite[Proof of Proposition 5.4]{Silverman_2009}. Thus, \(K_2(k;E_k)\) is \(p\)-divisible. Let \(\rho: X_k \to X_v\) denote the base change, and recall that the composition
\[
    F^2(X_v) \overset{\rho^*}{\to} F^2(X_k) \overset{\rho_*}{\to} F^2(X_v)
\] 
is given by multiplication by \([k:K_v]\). Since \(p\geq 5\) we see that \(p \nmid [k:K_v]\), so for \(z\in F^2(X_v)\) we have that \(z\) not \(p\)-divisible if \([k:K_v]z\) not \(p\)-divisible, which in turn implies that \(\rho^*(z)\) is not \(p\)-divisible. But \(F^2(X_k)/p = 0\), so it must be that \(F^2(X_v)/p = 0\) as well.
\end{pf}

We thus have the following reduction.

\begin{proposition}\label{prop:reduction_to_K2}
    Let \(E/K\) be an elliptic curve with complex multiplication, and let \(p\geq 5\) be a prime. Set \(X = E\times E\). Then the complex (\ref{eqn:p_primary_F2_complex}) is exact for this \(X/K\) and this \(p\) if and only if 
    \begin{align}\label{eqn:p_primary_K2_complex_narrow}
        \limi_n K_2(K;E) /p^n \overset{\Delta}{\rightarrow} \prod_{v\mid p} \limi_n K_2(K_v, E_v) /p^n \overset{\varepsilon}{\rightarrow} \Hom\left(\frac{\Br(X)\{p\}}{\Br_1(X)\{p\}}, \mathbb{Q}/\mathbb{Z}\right),
    \end{align}
    is exact.
\end{proposition}

\section{The adelic Albanese kernel completed at \texorpdfstring{\(p\)}{p}}\label{sec:local_structure}

Throughout this section, let \(E\) be an elliptic curve defined over a field \(k\) (the nature of \(k\) will vary), assume that \(E\) has complex multiplication by the full ring of integers \(\mathcal{O}_K\) for some imaginary quadratic field \(K\) of class number 1, and let \(p\) be a rational prime which splits in \(K/\mathbb{Q}\). Recall that there is a finite list of such fields \(K\), which are given by \(K = \mathbb{Q}(\sqrt{-D})\) for some \(D\in \{1,2,3,7,11,19,43,67,163\}\).

Our goal is to analyze the structure of the middle term of the complex (\ref{eqn:p_primary_K2_complex_narrow}) for \(E\) defined over the global fields constructed at the beginning of \autoref{ssec:curves_considered}. We will follow work of \cite{Hiranouchi_Hirayama_2013}, \cite{Hiranouchi_2016}, \cite{Gazaki_Leal_2021}, and \cite{Gazaki_Koutsianas_2024} to obtain an explicit description. The main tool will be the generalized Galois symbol.

The action of \(\End_k(E)\) on the space of invariant differentials of \(E\) defined over \(k\) gives a canonical identification of \(\End_k(E)\) with a subring \(\mathcal{O}_K \subset K \subset k\) \cite[Section 1]{Rubin_1999}, and for \(\alpha \in \mathcal{O}_K\) we write \([\alpha]\) to denote the corresponding endomorphism of \(E\). Since \(p\) splits in \(K/\mathbb{Q}\) and \(K\) has class number 1, we can write \(p = \pi \bar{\pi}\) for some irreducible element \(\pi\in \mathcal{O}_K\). This gives a corresponding factorization \([p] = [\pi][\bar{\pi}]\) in \(\End_k(E)\).

\subsection{Local decomposition of points}

Throughout this section, let \(k/\mathbb{Q}_p\) finite, and choose a minimal Weierstrass model for \(E\). We choose our factorization \(p = \pi\bar{\pi}\) to be such that the valuation on \(k\) restricts to that induced by \(\pi\) on \(K\). Recall that for any \(\ell/k\) finite we have a short exact sequence of \(G = \Gal(\bar{k}/k)\)-modules
\begin{align}\label{eqn:ses_for_local_ECs}
    0 \to \widehat{E}(\mathfrak{m}_\ell) \rightarrow E(\ell) \overset{r}{\rightarrow} \bar{E}(\mathbb{F}_\ell)\rightarrow 0
\end{align}
where \(r\) is the reduction map \cite[Proposition VII.2.1]{Silverman_2009}. Since \(p\) splits in \(K/\mathbb{Q}\), \(E\) has good ordinary reduction, i.e.\ the reduced curve \(\bar{E}\) is an ordinary elliptic curve over \(\mathbb{F}_k\). It follows from \cite{Deuring_1941} (see also \cite[Theorem 13.4.12]{Lang_1987}) that restriction gives a surjection \(\mathcal{O}_K = \End_{k}(E) \to \End_{\mathbb{F}_k}(\bar{E})\). Thus, there exists \(\widetilde{\pi}\in \mathcal{O}_K\) such that \([\widetilde{\pi}]\) reduces to the Frobenius endomorphism. In fact, since by \cite[Proposition 3.14]{Rubin_1999} we have that \([\widetilde{\pi}]\) acts on the formal group \(\widehat{E}\) by 
\begin{align*}
    [\widetilde{\pi}](Z) = \widetilde{\pi}Z + O(Z^2)
\end{align*}
and also that \([\widetilde{\pi}](Z) \equiv Z^q \pmod{\pi}\) by \cite[Corollary 3.9]{Rubin_1999}, we see that \(\widetilde{\pi}\) must be an associate of \(\pi\) in \(\mathcal{O}_K\).

It will be useful to us to extract the ``formal'' piece of a point \(P\in E(k)\), as in the following split short exact sequences.

\begin{lemma}\label{lem:ses_for_p_(co)torsion}
For any \(n\geq 1\) and any \(\ell/k\), we have split short exact sequences of \(G\)-modules
\begin{align}\label{eqn:ses_for_local_ECs_p_torsion}
    0 \to \widehat{E}[p^n] \rightarrow E[p^n] \overset{r}{\rightarrow} \bar{E}[p^n] \rightarrow 0
\end{align}
and
\begin{align}\label{eqn:ses_for_local_ECs_p_cotorsion}
    0 \to \widehat{E}(\mathfrak{m}_\ell)/p^n \rightarrow E(\ell)/p^n \overset{r}{\rightarrow} \bar{E}(\mathbb{F}_\ell)/p^n \rightarrow 0.
\end{align}
The short exact sequence (\ref{eqn:ses_for_local_ECs}) also splits, and this induces the splitting of (\ref{eqn:ses_for_local_ECs_p_cotorsion}).
\end{lemma}

\begin{pf}~(cf. \cite[Section 3.2.1]{Gazaki_Koutsianas_2024})
Consider the commutative diagram with exact rows
\[\begin{tikzcd}
	0 & {\widehat{E}(\mathfrak{m}_\ell)} & {E(\ell)} & {\bar{E}(\mathbb{F}_\ell)} & 0 \\
	0 & {\widehat{E}(\mathfrak{m}_\ell)} & {E(\ell)} & {\bar{E}(\mathbb{F}_\ell)} & 0
	\arrow[from=1-1, to=1-2]
	\arrow[from=1-2, to=1-3]
	\arrow["{p^n}"', from=1-2, to=2-2]
	\arrow["r", from=1-3, to=1-4]
	\arrow["{p^n}"', from=1-3, to=2-3]
	\arrow[from=1-4, to=1-5]
	\arrow["{p^n}"', from=1-4, to=2-4]
	\arrow[from=2-1, to=2-2]
	\arrow[from=2-2, to=2-3]
	\arrow["r", from=2-3, to=2-4]
	\arrow[from=2-4, to=2-5]
\end{tikzcd}\]
Applying the snake lemma gives an exact sequence of \(G\)-modules
\begin{align*}
    0 \to \widehat{E}[p^n](\mathfrak{m}_\ell) \rightarrow E[p^n](\ell) \overset{r}{\rightarrow} \bar{E}[p^n](\mathbb{F}_\ell) \rightarrow \widehat{E}(\mathfrak{m}_\ell)/p^n \rightarrow E(\ell)/p^n  \overset{r}{\rightarrow} \bar{E}(\mathbb{F}_\ell)/p^n \rightarrow 0.
\end{align*}
\noindent
Using \cite[Proposition 5.4]{Rubin_1999}, we have a decomposition
\begin{align*}
    E[p^n] \cong \mathcal{O}_K/(p^n) \cong \mathcal{O}_K/(\pi^n) \oplus \mathcal{O}_K/(\bar{\pi}^n) \cong E[\pi^n] \oplus E[\bar{\pi}^n],
\end{align*}
and since the endomorphisms \([\pi]\) and \([\bar{\pi}]\) are defined over \(K\subset k\) this decomposition is as \(G\)-modules as well. Noting that \([\pi^n]\) is of degree \(p^n\) and recalling that there is an associate \(\widetilde{\pi}\) of \(\pi\) such that \([\widetilde{\pi}]\) restricts to the Frobenius endomorphism, we see that \(\widehat{E}[p^n]\) coincides with \(E[\pi^n]\) as \(G\)-submodules of \(E[p^n]\). Thus, the restriction map \(r\) vanishes on \(E[\pi^n]\). Since \(r\) is surjective for any \(\ell\) such that \(E[p^n] \subseteq E(\ell)\), we see that \(E[\bar{\pi}^n]\) is mapped isomorphically onto \(\bar{E}[p^n]\) as \(G\)-modules, and the sequence (\ref{eqn:ses_for_local_ECs_p_torsion}) splits canonically.

Now consider the short exact sequence (\ref{eqn:ses_for_local_ECs}). Since \(E\) has good reduction, by \cite[Proposition VII.3.1]{Silverman_2009} for \(m\) coprime to \(p\) the reduction map is injective on \(m\)-torsion points and by \cite[Proposition VII.4.1]{Silverman_2009} \(E[m]\) is unramified, so \(r\) induces an isomorphism from the coprime-to-\(p\) torsion subgroup of \(E(\ell)\) to that of \(\bar{E}(\ell)\). Then since \(\bar{E}(\ell)\) is finite, it suffices to show that \(r\) gives an isomorphism \(E(\ell)\{p\}\cong \bar{E}(\ell)\{p\}\). Since \(E\) has ordinary reduction, we may write \(\bar{E}(\ell)\{p\} = \bar{E}[p^{n_0}]\) for some \(n_0 \geq 0\), and applying the above for \(n_0\) we have that \(r\) induces a \(G\)-module isomorphism \(E[\bar{\pi}^{n_0}] \cong \bar{E}[p^{n_0}]\). Thus, \(r\) restricts to an isomorphism
\begin{align*}
    E[\bar{\pi}^{n_0}] \oplus E(\ell)\{m\} \cong \bar{E}(\ell),
\end{align*}
and so the sequence (\ref{eqn:ses_for_local_ECs}) splits as \(G\)-modules. As a consequence, the sequence (\ref{eqn:ses_for_local_ECs_p_cotorsion}) splits as well by tensoring with \(\mathbb{Z}/p^n\).\qedhere
\end{pf}

\begin{remark}\label{rmk:splitting_of_Mackey_cotorsion}
    The associations \(\ell\mapsto \widehat{E}(\mathfrak{m}_\ell)\) and \(\ell \mapsto \bar{E}(\mathbb{F}_\ell)\) together with the appropriate restriction and corestriction maps define Mackey functors \(\widehat{E}\) and \([E/\widehat{E}]\) respectively (see \cite[Section 3.3]{Raskind_Spiess_2000}). One can show that the splittings of (\ref{eqn:ses_for_local_ECs_p_cotorsion}) for various \(\ell\) are compatible with these restriction and corestriction maps, and so \autoref{lem:ses_for_p_(co)torsion} induces a split short exact sequence of Mackey functors
    \begin{align*}
        0 \to \widehat{E}/p^n \to E/p^n \to [E/\widehat{E}]/p^n \to 0
    \end{align*}
    for any \(n\). 
\end{remark}

Given a finite extension \(\ell/k\) and a point \(P\in E(\ell)\), we denote by \((\widehat{P}, \bar{P})\) its image under the isomorphism \(E(\ell) \cong \widehat{E}(\mathfrak{m}_\ell) \oplus \bar{E}(\mathbb{F}_\ell)\) induced by the splitting of (\ref{eqn:ses_for_local_ECs}); by an abuse of notation, we will write \(P\mapsto (\widehat{P}, \bar{P})\) under the splitting of (\ref{eqn:ses_for_local_ECs_p_cotorsion}) for \(n=1\) as well. 

Letting \(u\) denote the valuation on \(\ell\), recall that \(u\) induces a valuation on \(\widehat{E}(\mathfrak{m}_\ell)\subset E(\ell)\) via the identification
\begin{align*}
    z \mapsto \left(\frac{z}{w(z)} , \frac{-1}{w(z)} \right)
\end{align*}
where \(w(z) = z^3 + O(z^4)\) \cite[Proposition VII.2.2]{Silverman_2009}. We then see that \(P \in \widehat{E}(\mathfrak{m}_\ell)\) if and only if \(u(x(P)) < 0\) (or equivalently, \(u(y(P)) < 0\)), and that in this case \[u(P) = \frac{-u(x(P))}{2} = \frac{-u(y(P))}{3}.\]

\begin{remark}\label{rmk:point_valuations_same_after_coprime_to_p_scaling}
    This valuation corresponds to the filtration on \(\widehat{E}(\mathfrak{m}_\ell)\) given by \[\widehat{E}(\mathfrak{m}^i_\ell) = \{P\in \widehat{E}(\mathfrak{m}_\ell) \mid u(P) \geq i\}.\] All of the successive quotients of this filtration are \(\mathbb{F}_\ell\) by \cite[Proposition IV.3.2]{Silverman_2009}. Consequently, we see that for any \(P \in \widehat{E}(\mathfrak{m}_\ell)\), \(u(dP) = u(P)\) for any \(d\) coprime to \(p\).
\end{remark}

The following key Lemma will allow for us to make certain non-triviality criteria very explicit in \autoref{sec:non-trivial_symbols}.

\begin{lemma}\label{lem:P_hat_computation}
    Write \(|\bar{E}(\mathbb{F}_\ell)| = dp^n\) with \(p\nmid d\). Then \(E[\bar{\pi}^n] \subseteq E(\ell)\). If \(p\neq 2\), then for any \(P\in E(\ell) \setminus \widehat{E}(\mathfrak{m}_\ell)\) there exists a unique \(T \in E[\bar{\pi}^n]\) such that \(u(x(dP) - x(T)) \geq 1\) and \(u(y(dP) - y(T)) = 0\). 
    
    For this choice of \(T\), we have \(\widehat{dP} = dP + T\) with \(u(\widehat{dP}) = u(x(dP) - x(T))\).
\end{lemma}
\begin{proof}
    That \(E[\bar{\pi}^n] \subseteq E(\ell)\) follows from the \(G\)-module isomorphism of \(E[\bar{\pi}^n]\) with \(\bar{E}[p^n]\).

    Suppose that \(P \in E[\bar{\pi}^n]\). Note that \(u(y(T)) = 0\) for all \(T \in E[\bar{\pi}^n] \setminus \{O_E\}\), since otherwise the torsion point \(T\) of order \(p\neq 2\) would reduce to a 2-torsion point in \(\bar{E}(\mathbb{F}_\ell)\). Then \(T = -dP\) is as desired under the convention that \(u(0) = +\infty\), since \(x(dP) - x(-dP) = 0\) and \(p\neq 2\) implies that \(u(y(dP) - u(-dP)) = u(2y(dP)) = u(y(dP))\). Also, it is clear by definition that \(\widehat{dP} = O_E = dP + T\).  Now, suppose for contradiction that \(T\in E[\bar{\pi}^n]\setminus \{-dP, O_E\}\) satisfies the given valuative conditions. It cannot be that \(T = dP\), as then it would hold that \(u(y(dP) - y(T)) = +\infty\). Thus, the addition formula on \(E\) gives that
    \begin{align}\label{eqn:point_addition_formula}
        x(dP + T) = \frac{y(dP) - y(T)}{x(dP) - x(T)} + x(dP) + x(T).
    \end{align}
    Since \(dP, T \in E(\ell) \setminus \widehat{E}(\mathfrak{m}_\ell)\) have \(x\)-coordinates of positive valuation, our assumptions give that \(u(x(dP + T)) < 0\) and thus \(\widehat{dP} = dP + T\), which contradicts the well-definedness of \(\widehat{dP}\) implied by \autoref{lem:ses_for_p_(co)torsion}.
    
    Now, suppose that \(P \in E(\ell) \setminus \widehat{E}(\mathfrak{m}_\ell)\) is not \(\bar{\pi}\)-torsion. \autoref{lem:ses_for_p_(co)torsion} implies that there exists a unique \(T\in E[\bar{\pi}^n]\) such that \(dP + T \in \widehat{E}(\mathfrak{m}_\ell)\), so it suffices to check that the valuative conditions hold for this \(T\). Once again, we consider the point addition formula (\ref{eqn:point_addition_formula}). Since \(dP + T \in \widehat{E}(\mathfrak{m}_\ell)\) it must be that \(u(x(dP + T)) < 0\). Noting that we again have that the coordinates of \(dP\) and \(T\) are of non-negative valuation, it must be the case that \(u(x(dP) - x(T)) > 0\). Further, if it were the case that \(u(y(dP) - y(T))\) were positive, then since \(p \neq 2\) we compute that \(u(x(dP - T)) < u(x(dP + T))\), contradicting uniqueness of \(T\). Thus, it must be that \(u(y(dP) - y(T)) = 0\), and we conclude by noting that \(u(\widehat{dP}) = u(dP + T) = -2u(x(dP) - x(T))\), as desired.
\end{proof}

\subsection{A filtration on the quotient \texorpdfstring{\(\widehat{E}(\mathfrak{m})/p\)}{E\_hat(m)/p}}

Let \(k/\mathbb{Q}_p\) finite with valuation \(v\), and assume that \(E[p]\subseteq E(k)\). This implies that \(\widehat{E}[p] \cong \mu_p\) as \(\Gal(\bar{k}/k)\)-modules, since \(\mu_p\subseteq k\) (as in the previous section, \(p\) splitting in \(K/\mathbb{Q}\) implies that \(E\) has ordinary reduction, so \(\widehat{E}[p] \cong \mathbb{Z}/p\) as abelian groups). Also note that the absolute ramification index of \(k\) is divisible by \(p-1\) for the same reason.

The filtration \((\widehat{E}(\mathfrak{m}_k^i))_{i\geq 1}\) on \(\widehat{E}(\mathfrak{m}_k)\) discussed in the previous section induces a filtration \((\mathcal{D}^i_k)_{i\geq 1}\) on \(\widehat{E}(\mathfrak{m}_k)/p\) by letting \(\mathcal{D}^i_k\) denote the image of \(\widehat{E}(\mathfrak{m}^i_k)\) under the quotient map \(q\). Let \(\widetilde{v}:\widehat{E}(\mathfrak{m}_k)/p \to \mathbb{Z}_{\geq 1} \cup \{\infty\}\) be the corresponding valuation, i.e. the function
\begin{align*}
    \widetilde{v}(q(P)) = \sup \{i \mid q(P) \in \mathcal{D}^i_k\}.
\end{align*}

The following Lemma will justify our using \(v\) in the sequel to denote the valuation on both \(\widehat{E}(\mathfrak{m}_k)\) and \(\widehat{E}(\mathfrak{m}_k)/p\).

\begin{lemma}\label{lem:point_valuations_are_the_same_mod_p}
    For any \(P \in \widehat{E}(\mathfrak{m}_k)\) with \(v(P) \leq p-1\), we have \(\widetilde{v}(q(P)) = v(P)\).
\end{lemma}
\begin{proof}
    By construction, for any \(i\geq 1\) we have a surjective map
    \begin{align*}
        \mathbb{F}_k \cong \widehat{E}(\mathfrak{m}^i_k) / \widehat{E}(\mathfrak{m}^{i+1}_k) \to \mathcal{D}^i_k / \mathcal{D}^{i+1}_k.
    \end{align*}
    When \(i \leq p-1\), the codomain is also isomorphic to \(\mathbb{F}_k\) by \cite[Lemma 2.1.4]{Kawachi_2002}. Thus, the given map is an isomorphism for these \(i\), and the statement follows.
\end{proof}

Consider the connecting map
\begin{align*}
    \delta: \widehat{E}(\mathfrak{m}_k)/p \to H^1(k, \widehat{E}[p]) \cong H^1(k, \mu_p) \cong k^\times / p
\end{align*}
arising from the Kummer short exact sequence for the multiplication-by-\(p\) endomorphism on \(\widehat{E}\). By \cite[Section 2, p. 251]{Kawachi_2002}, this map can be defined by mapping a point \(P\) to \(\alpha \in k^\times\) such that \(k(Q) = k(\sqrt[p]{\alpha})\), where \([p]Q = P\). 

Letting \(U_k = 1 + \mathfrak{m}_k\), we note that we have a short exact sequence
\[\begin{tikzcd}
	0 & {U_k} & {\mathcal{O}_k^\times} & {\mathbb{F}_k^\times} & 0.
	\arrow[from=1-1, to=1-2]
	\arrow[from=1-2, to=1-3]
	\arrow[from=1-3, to=1-4]
	\arrow[from=1-4, to=1-5]
\end{tikzcd}\]
Tensoring with \(\mathbb{Z}/p\) then gives an isomorphism \(U_k/p \cong \mathcal{O}_k^\times / p\), as the \(p^{\text{th}}\)-power Frobenius map is an isomorphism on \(\mathbb{F}_k^\times\). Further, we recall that \(U_k\) comes equipped with a filtration \(U_k^i = 1 + \mathfrak{m}_k^i\), and that this induces a corresponding filtration \(\bar{U_k^i}\) on \(U_k/p\) by again taking images under the quotient map. The following Theorem shows that \(\delta\) respects the filtrations on its domain and codomain.

\begin{theorem}\label{thm:delta_filtered_isomorphism}
    \cite[Theorem 2.1.6]{Kawachi_2002}
    The map \(\delta\) is an isomorphism. Further, \(\delta(\mathcal{D}^i_k) \subseteq \bar{U}^i_k\) for all \(i\), and \(\delta\) induces isomorphisms on successive quotients
    \[
        \mathcal{D}^i_k/\mathcal{D}^{i+1}_k \cong \bar{U}^i_k/\bar{U}^{i+1}_k.
    \] 
\end{theorem}

Recall that we have a Mackey functor \(\mathbb{G}_m\) which associates to any \(\ell/k\) the unit group \(\ell^\times\). For each \(i\geq 0\), the subgroups \(U^i_k\) are compatible with norms and restrictions, and so descibe a sub-Mackey functor \(U^i\) of \(\mathbb{G}_m\). Passing to quotients, we obtain for each \(i\) a sub-Mackey functor \(\bar{U}^i \subseteq \mathbb{G}_m/p\).

\begin{corollary}\label{cor:delta_Mackey_isomorphism}
    If \(k\) is as above, \(\delta\) gives an isomorphism of Mackey functors \(\widehat{E}/p \cong \bar{U}^1\), and this isomorphism is again compatible with filtrations.
\end{corollary}

\subsection{Somekawa \texorpdfstring{\(K\)}{K}-groups of elliptic curves over local fields mod \texorpdfstring{\(p\)}{p}}
Once again, let \(k/\mathbb{Q}_p\) finite.
Our goal here is to follow a proof determining the structure of the Somekawa \(K\)-group \(K_2(k;E)/p\), and in doing so determine a non-triviality criterion for symbols defined over \(k\). This proof is contained in its entirety in \cite[Section 3.2]{Gazaki_Koutsianas_2024}, and versions or essential ingredients of it can be found in \cite{Raskind_Spiess_2000}, \cite{Hiranouchi_2016}, \cite{Hiranouchi_Hirayama_2013}, \cite{Yamazaki_2005}, and \cite{Kawachi_2002}.

\begin{theorem}\label{thm:K2_mod_p_structure}
Let \(E/k\) an elliptic curve with complex multiplication by some \(\mathcal{O}_K\), let \(p\) a prime splitting in \(K/\mathbb{Q}\), and assume that \(E[p]\subseteq E(k)\). Then \(K_2(k;E)/p\cong \Br(k)[p]\) via
\begin{align*}
    \{P, Q\}_{k/k} \mapsto \cor_{\ell/k} \big( \delta(\widehat{P}), \delta(\widehat{Q})\big)_\zeta,
\end{align*}
where \(\delta : \widehat{E}(k)/p \to k^\times / p\) is the connecting map from the Kummer sequence for \(\widehat{E}(k)\), \(\zeta\) is a primitive \(p\)-th root of unity, and \((-,-)_\zeta\) denotes the Hilbert symbol.
\end{theorem}
\begin{pf}
Recall that the generalized Galois symbol \(s_p:K_2(k;E)/p \to H^2(k, E[p]^{\otimes 2})\), for all \(\ell/k\) finite, fits into the commutative diagram
\[\begin{tikzcd}
	{(E(\ell)/p)^{\otimes 2}} && {K_2(k;E)/p} \\
	{H^1(\ell,E[p])^{\otimes 2}} && {H^2(k, E[p]^{\otimes 2})}
	\arrow[from=1-1, to=1-3]
	\arrow["{\delta^{\otimes 2}}"', from=1-1, to=2-1]
	\arrow["{s_p}", from=1-3, to=2-3]
	\arrow["{\cor_{\ell/k}\circ \smile}"', from=2-1, to=2-3]
\end{tikzcd}\]
Since the connecting homomorphism \(\delta\) is compatible with norms and restrictions, we can instead characterize \(s_p\) with a single commutative diagram using Mackey functors:

\[\begin{tikzcd}
	{(E/p)^{\mprod 2}(k)} && {K_2(k;E)/p} \\
	{H^1(-,E[p])^{\mprod 2}(k)} && {H^2(k, E[p]^{\otimes 2})}
	\arrow[from=1-1, to=1-3]
	\arrow["{\delta^{\otimes 2}}"', from=1-1, to=2-1]
	\arrow["{s_p}", from=1-3, to=2-3]
	\arrow["{\cor_{-/k}\circ \smile}"', from=2-1, to=2-3]
\end{tikzcd}\]

The first step is to reduce the problem to only the formal groups, using the decomposition of Mackey functors \(E/p \cong \widehat{E}/p \oplus [E/\widehat{E}]/p\) from \autoref{rmk:splitting_of_Mackey_cotorsion}.

\begin{lemma}
    (See, e.g., \cite[Section 3]{Takemoto_2011})
    The composition
    \[\begin{tikzcd}
        {(E/p)^{\mprod 2}(k)} & {H^1(-,E[p])^{\mprod 2}(k)} && {H^2(k, E[p]^{\otimes 2})}
        \arrow["{\delta^{\otimes 2}}"', from=1-1, to=1-2]
        \arrow["{\cor_{-/k}\circ \smile}"', from=1-2, to=1-4]
    \end{tikzcd}\]
    decomposes as the direct sum of the compositions
    \[\begin{tikzcd}
        {(\widehat{E}/p\mprod \widehat{E}/p)(k)} & {\big(H^1(-,\widehat{E}[p])\mprod H^1(-,\widehat{E}[p])\big)(k)} && {H^2(k, \widehat{E}[p]\otimes \widehat{E}[p])} \\
        {(\widehat{E}/p\mprod [E/\widehat{E}]/p)(k)} & {\big(H^1(-,\widehat{E}[p])\mprod H^1(-,\bar{E}[p])\big)(k)} && {H^2(k, \widehat{E}[p]\otimes \bar{E}[p])} \\
        {([E/\widehat{E}]/p\mprod \widehat{E}/p)(k)} & {\big(H^1(-,\bar{E}[p])\mprod H^1(-,\widehat{E}[p])\big)(k)} && {H^2(k, \bar{E}[p]\otimes \widehat{E}[p])} \\
        {([E/\widehat{E}]/p\mprod [E/\widehat{E}]/p)(k)} & {\big(H^1(-,\bar{E}[p])\mprod H^1(-,\bar{E}[p])\big)(k)} && {H^2(k, \bar{E}[p]\otimes \bar{E}[p])}
        \arrow["{\delta^{\otimes 2}}"', from=1-1, to=1-2]
        \arrow["{\cor_{-/k}\circ \smile}"', from=1-2, to=1-4]
        \arrow["{\delta^{\otimes 2}}"', from=2-1, to=2-2]
        \arrow["{\cor_{-/k}\circ \smile}"', from=2-2, to=2-4]
        \arrow["{\delta^{\otimes 2}}"', from=3-1, to=3-2]
        \arrow["{\cor_{-/k}\circ \smile}"', from=3-2, to=3-4]
        \arrow["{\delta^{\otimes 2}}"', from=4-1, to=4-2]
        \arrow["{\cor_{-/k}\circ \smile}"', from=4-2, to=4-4]
    \end{tikzcd}\]
    where the maps \(\delta\) are the appropriate connecting homomorphisms from the Kummer sequences on \(\widehat{E}\) and \(\bar{E}\).
\end{lemma}

Since the left-most term of the last three of these compositions vanishes by \cite[Proof of Theorem 3.14]{Gazaki_Leal_2021} and \cite[Lemma 3.4.2]{Raskind_Spiess_2000}, we see that \(s_p\) must have image entirely contained in \(H^2(k, \widehat{E}[p]^{\otimes 2})\) and fits into the commutative diagram
\[\begin{tikzcd}
	{(\widehat{E}/p)^{\mprod 2}(k)} && {K_2(k;E)/p} \\
	{H^1(-,\widehat{E}[p])^{\mprod 2}(k)} && {H^2(k, \widehat{E}[p]^{\otimes 2})}
	\arrow[from=1-1, to=1-3]
	\arrow["{\delta^{\otimes 2}}"', from=1-1, to=2-1]
	\arrow["{s_p}", from=1-3, to=2-3]
	\arrow["{\cor_{-/k}\circ \smile}"', from=2-1, to=2-3]
\end{tikzcd}\]

The key idea is to now relate the cup product above to the usual Hilbert symbol. Since \(E[p]\subseteq E(k)\) and thus \(\mu_p\subseteq k\), fixing a \(\Gal(\bar{k}/k)\)-module isomorphism \(\widehat{E}[p]\cong \mu_p\) gives corresponding isomorphisms of Mackey functors \(H^1(-, \widehat{E}[p]) \cong H^1(-, \mu_p) \cong \mathbb{G}_m/p\) and \(H^2(k, \widehat{E}[p]^{\otimes 2}) \cong H^2(k, \mu_p\otimes \mu_p)\). Fixing a primitive \(p\)-th root of unity \(\zeta\) defines an isomorphism \(H^2(k, \mu_p \otimes \mu_p) \cong H^2(k, \mu_p) = \Br(k)[p]\). By \cite[Chapter XIV Section 2]{Serre_1979}, we then have that under these isomorphisms, for any \(\ell/k\) finite the cup product on \(H^1(\ell, \mu_p)^{\otimes 2}\) corresponds to the Hilbert symbol on \((\ell^{\times}/p)^{\otimes 2}\) taking \(\alpha\otimes \beta\) to the cyclic algebra \((\alpha,\beta)_\zeta\) as defined in \cite[Chapter 15]{Milnor_1972}.

By \autoref{cor:delta_Mackey_isomorphism}, we have that the image of \(\delta:\widehat{E}/p \to \mathbb{G}_m/p\) is the sub-Mackey functor \(\bar{U}^1 \cong \bar{U}^0\). We may then conclude using \cite[Lemma 3.3]{Hiranouchi_2016} that the Hilbert symbol gives an isomorphism \((\bar{U}^0 \mprod \bar{U}^0)(k) \cong \Br(k)[p]\). Thus, the generalized Galois symbol is an isomorphism \(K_2(k;E) \cong \Br(k)[p]\), and this isomorphism is given by the map
\[
\{P,Q\}_{\ell/k} \mapsto \cor_{\ell/k} \big( \delta(\widehat{P}), \delta(\widehat{Q})\big)_\zeta. \qedhere
\]
\end{pf}

Recall that for \(\alpha,\beta\in k\), the cyclic algebra \((\alpha,\beta)_\zeta\) is trivial in \(\Br(k)\) if and only if \(\beta\) is a norm from \(k(\sqrt[p]{\alpha})\) \cite[Theorem 15.7]{Milnor_1972}.

Using the isomorphism in the above theorem, we obtain a criterion for checking non-triviality of local symbols via the valuations of the formal components of the points involved, provided that \(k/\mathbb{Q}_p\) has no inertia.

\begin{lemma}\label{lem:nontriviality_via_valuations}
    Let \(k/\mathbb{Q}_p\) finite with valuation \(v\), and let \(E/k\) an elliptic curve of good ordinary reduction with CM by some \(\mathcal{O}_K\) such that \(E[p]\subseteq E(k)\).

    Let \(P,Q\in E(k)\). If \(k/\mathbb{Q}_p\) is totally ramified of degree \(p-1\) and \(v(\widehat{P}) + v(\widehat{Q}) = p\), then \(\{P,Q\}_{k/k}\) is non-trivial modulo \(p\).
\end{lemma}
\begin{pf} (Cf. \cite[Theorem 4.7, Proof of Claim 1]{Gazaki_Koutsianas_2024})
Let \(i = v(\widehat{P})\) and \(j = v(\widehat{Q})\), and fix a primitive \(p\)-th root of unity \(\zeta\).
By \autoref{lem:point_valuations_are_the_same_mod_p}, \autoref{thm:delta_filtered_isomorphism} and the isomorphism in \autoref{thm:K2_mod_p_structure}, it suffices to show that the cyclic algebra \((\alpha, \beta)_\zeta\) is non-trivial for \(\alpha = \delta(\widehat{P}) \in \bar{U}_{k}^i\setminus \bar{U}_{k}^{i+1}\) and \(\beta = \delta(\widehat{Q}) \in \bar{U}_{k}^j\setminus \bar{U}_{k}^{j+1}\), i.e.\ that \(\beta \not\in N_{k(\sqrt[p]{\alpha})/k}(k(\sqrt[p]{\alpha})^\times)\). 

As before, \(k(\sqrt[p]{\alpha}) = k([p]^{-1}\widehat{P})\), where \([p]:\widehat{E} \to \widehat{E}\) is the multiplication-by-\(p\) formal group isogeny. Since this isogeny can be written as \([p](T) = a_1T + O(T^2)\) with \(a_1 = p\), we see that \([p]\) has height 1, and the value \(t = v(a_1)/p-1\) appearing in \cite[Lemma 2.1.5]{Kawachi_2002} is given by \(t=1\). Thus, since \(1\leq i<p\) the extension \(k(\sqrt[p]{\alpha})/k\) is cyclic and totally ramified of degree \(p\), and the jump in the ramification filtration on \(\Gal(k(\sqrt[p]{\alpha})/k)\) happens at \(s=p-i\). We can then conclude using \cite[Section V.3]{Serre_1979} that we have an isomorphism 
\[U_k^{p-i} / N(U^{p-i}_{k(\sqrt[p]{\alpha})}) \cong k^{\times}/N(k(\sqrt[p]{\alpha})^\times) \underset{\sim}{\xrightarrow{(\alpha,-)_\zeta}} \Br(k)[p]\] 
which induces a surjective map
\[
    U_k^{p-i}/U_k^{p-i+1} = U_k^{p-i} / N(U^{p-i+1}_{k(\sqrt[p]{\alpha})}) \twoheadrightarrow \Br(k)[p].
\]
This map then factors through $\bar{U}_k^{p-i}/\bar{U}_k^{p-i+1}$, which by \cite[Lemma 2.1.4]{Kawachi_2002} is isomorphic to $\mathbb{F}_k$. Noting that both $\mathbb{F}_k$ and $\Br(k)[p]$ are of order $p$ (the former by the assumption that $k/\mathbb{Q}_p$ is totally ramified), we get that $\bar{U}_k^{p-i}/\bar{U}_k^{p-i+1} \to \Br(k)[p]$ is in fact an isomorphism, and conclude that any element of \(\bar{U}_k^{p-i}\setminus \bar{U}_k^{p-i+1}\) pairs non-trivially with \(\alpha\). Since \(i+j = p\) and \(\beta \in \bar{U}_k^{j}\setminus \bar{U}_k^{j+1}\), we see that \(\beta\) is exactly such an element.
\end{pf}

\begin{remark}
    The assumption that \(k/\mathbb{Q}_p\) has no inertia is currently essential to guarantee non-triviality; if there is any inertia at all, the residue field \(\mathbb{F}_k\) is no longer 1-dimensional, and the kernel of \(\paren{\alpha,-}_\zeta\) has non-trivial intersection with \(U_k^{p-i}\setminus U_k^{p-i+1}\). The exact nature of this kernel would need to be unravelled in order to get an analogous criterion in this case.
\end{remark}

\subsection{The adelic Albanese kernel for products of elliptic curves}\label{ssec:curves_considered}

Now suppose that \(E\) is defined over the field \(K\) itself. Here, the assumptions that \(E\) has complex multiplication by \(\mathcal{O}_K\) and \(p\) is a prime splitting in \(K/\mathbb{Q}\) imply that \(E\) has good ordinary reduction at either of the places \(v\mid p\) of \(K\). For \(F/K\) a finite extension, we can associate to \(F\), \(E\), and \(v\) a field extension \(L/F\) defined by \(L = F(E[\pi])\), where \(\pi\) generates the ideal of \(\mathcal{O}_K\) corresponding to \(v\). 

\begin{lemma}\label{lem:places_of_L_above_p_structure}~
    \samepage
    \begin{enumerate} 
        \item Suppose that \(p\geq 5\). Then the extension \(K(E[\pi])/K\) is of degree \(p-1\), is totally ramified at \(v\), and is unramified at all other places of \(K\).
        \item The places of \(F\) lying above \(v\) are totally ramified in \(L/F\), and all other places of \(F\) are unramified in \(L/F\). If \(w\) is a place of \(L\) lying above \(v\), then its absolute ramification index is divisible by \(p-1\).
    \end{enumerate}
\end{lemma}
\begin{proof}
    Note that (2) follows immediately from (1) by routine algebraic number theory.

    By \cite[Corollary 5.20.(ii)]{Rubin_1999}, we have an isomorphism \[\Gal(K(E[\pi])/K) \to \big(\mathcal{O}_K/(\pi)\big)^\times \cong \mathbb{F}_p^\times,\] so \(K(E[\pi])/K\) has degree \(|\mathbb{F}_p^\times| = p-1\). Total ramification at \(v\) follows from part (iv) of the same result. For any other place \(v'\) of \(K\), note that \(\pi \in \mathcal{O}_{K_{v'}}^\times\). Then \(K(E[\pi])/K\) is unramified at \(v'\) if and only if \(K_{v'}(E_{v'}[\pi])/K_{v'}\) is unramified, and this holds by \cite[Corollary 3.17]{Rubin_1999}.
\end{proof}

We wish to investigate the behavior of the complex appearing in the statement of \autoref{prop:reduction_to_K2} for the curve \(E_L\) and the prime \(p\). An important tool for understanding the middle term of this complex is the following:

\begin{proposition}\label{prop:no_torsion_implies_divisible}
    Let \(k/\mathbb{Q}_p\) finite, and let \(E/k\) an elliptic curve with good ordinary reduction and complex multiplication. If \(E[p]\) is not \(k\)-rational, then \(K_2(k;E)\) is \(p\)-divisible.
\end{proposition}
\begin{pf}
    See \cite[Proof of Proposition 4.4]{Gazaki_Koutsianas_2024}.
\end{pf}

Combining these with the results of the previous section, we obtain the following:

\begin{theorem}\label{thm:nontriviality_criterion}
    Let \(E/K\) be an elliptic curve with complex multiplication by \(\mathcal{O}_K\), let \(v\mid p\) be a place of good ordinary reduction for \(E\) with \(p\geq 5\), and let \(F/K\) be a finite extension. Construct \(L/F\) as above, and let \(w\mid p\) a place of \(L\).

    \begin{enumerate}
        \item For \(w\mid \bar{v}\), if \(e_{\bar{v}}(F/K) < p-1\) then
        \begin{align*}
            \limi_n K_2(L_w;E_w)/p^n = 0
        \end{align*}
        \item For \(w\mid v\), let \(u\) the place of \(F\) lying below \(w\). If \(p \nmid |\bar{E_u}(\mathbb{F}_u)|\), then
        \begin{align*}
            \limi_n K_2(L_w;E_w)/p^n = 0.
        \end{align*}
        Otherwise, we have an isomorphism
        \begin{align*}
            K_2(L_w;E_w)/p &\cong H^2(L_w,\mu_p) = \Br(L_w)[p] \cong \mathbb{Z}/p\\
        \curly{P,Q}_{k/L_w} &\mapsto \cor_{k/L_w}\paren{\delta(\widehat{P}), \delta(\widehat{Q})}_\zeta.
        \end{align*}
        If \(e_v(F/K) < p\), then
        \begin{align*}
            \limi_n K_2(L_w;E_w)/p^n = K_2(L_w;E_w)/p \cong \mathbb{Z}/p.
        \end{align*}
    \end{enumerate}
\end{theorem}
\begin{pf}
Fix a place \(w\mid \bar{v}\), and suppose that \(e_{\bar{v}}(F/K) < p-1\). Since \(\mathbb{Q}_p(\mu_p)/\mathbb{Q}_p\) has ramification index \(p-1\) and by \autoref{lem:places_of_L_above_p_structure} \(K(E[\pi])/K\) is unramified at \(\bar{v}\), we see that \(\mu_p \not\subset L_w\). Thus, \(E_w[p]\) is not \(L_w\)-rational, and by \autoref{prop:no_torsion_implies_divisible} we have \(\limi_n K_2(L_w;E_w)/p^n = 0\).

Now suppose that \(w\mid v\). Since \(E\) has good ordinary reduction and complex multiplication, by \autoref{lem:ses_for_p_(co)torsion} we have a \(\Gal(\bar{L_w}/L_w)\)-module isomorphism \(E_w[p] \cong \widehat{E_w}[p] \oplus \bar{E_w}[p]\). Examining the proof of this Lemma, we note that \(\widehat{E_w}[p]\) coincides with \(E_w[\pi]\), which is \(L\)-rational (and thus \(L_w\)-rational) by construction.  It follows that \(E_w[p]\) is \(L_w\)-rational if and only if \(\bar{E_w}[p]\) is, and since \(E\) has good ordinary reduction this is equivalent to having \(p \mid |\bar{E_w}(\mathbb{F}_w)|\). By \autoref{lem:places_of_L_above_p_structure}, the extension \(L/F\) is totally ramified at \(u\), so \(\mathbb{F}_w \cong \mathbb{F}_u\) and we have that \(|\bar{E_w}(\mathbb{F}_w)| = |\bar{E_u}(\mathbb{F}_u)|\). Thus, if \(p \nmid |\bar{E_u}(\mathbb{F}_u)|\), then \(E_w[p]\) is not \(L_w\)-rational and we again have that \(\limi_n K_2(L_w;E_w)/p^n = 0\). Otherwise, \autoref{thm:K2_mod_p_structure} gives the desired isomorphism
\begin{align*}
    K_2(L_w;E_w)/p &\cong \Br(L_w)[p] \cong \mathbb{Z}/p.
\end{align*}

Finally, we note that since \(\mathbb{Q}_p(\mu_{p^2})\) has absolute ramification index divisible by \(p\), the assumption that \(e_v(F/K) < p\) implies that \(L_w(E_w[p^2]) / L_w\) is wildly ramified. We then have by \cite[Theorem 3.14]{Gazaki_Leal_2021} that \(p K_2(L_w;E_w)\) is \(p\)-divisible, and so \(\limi_{n}K_2(L_w;E_w)/p^n = K_2(L_w;E_w)/p\).
\end{pf}

\begin{corollary}\label{cor:size_of_middle_term}
    Let \(E/K\) be an elliptic curve with complex multiplication by \(\mathcal{O}_K\), let \(v\mid p\) a place of \(K\) of good ordinary reduction for \(E\) with \(p\geq 5\), and let \(F/K\) be an extension of degree \([F:K] < p-1\). Suppose that \(p \mid |\bar{E_v}(\mathbb{F}_v)|\). Construct \(L\) as above, and let \(X = (E\times E)_L\) Then
    \begin{align*}
        \limi_n F^2_{\mathbb{A}}(X)/p^n \cong \prod_{w\mid v} \limi_n K_2(L_w; E_w)/p^n \cong (\mathbb{Z}/p)^m,
    \end{align*}
    where \(m\) is the number of places of \(F\) lying above \(v\).
\end{corollary}

\begin{remark}
    This result contains a corrected version of \cite[Theorem 4.2.1]{Gazaki_Koutsianas_2024}, which incorrectly states that \(\limi_n K_2(L_w;E_w)/p^n \cong \mathbb{Z}/p\) for places lying above \(\bar{v}\).
\end{remark}

Additionally, under the assumptions of \autoref{cor:size_of_middle_term} we have by \cite[Theorem 4.2, Claim 3]{Gazaki_Koutsianas_2024} that the final term of the complex (\ref{eqn:p_primary_K2_complex_narrow}) vanishes:
\begin{align*}
    \Hom\left(\frac{\Br(X)\{p\}}{\Br_1(X)\{p\}}, \mathbb{Q}/\mathbb{Z}\right) = 0.
\end{align*}
Combining this with the above Theorem, we get:
\begin{corollary}\label{cor:adelic_F2_structure}
    Let \(E/K\) be an elliptic curve with complex multiplication by \(\mathcal{O}_K\), let \(v \mid p\) a place of \(K\) of good ordinary reduction for \(E\) with \(p\geq 5\), and let \(F/K\) be a finite extension with \([F:K] < p-1\). Construct \(L/F\) as above, and suppose that \(p \mid |E_v(\mathbb{F}_v)|\).

    The complex (\ref{eqn:p_primary_F2_complex}) is exact for 
    \(X = (E\times E)_L\) if and only if \(\Delta\) is surjective, which is equivalent to having \(m\) global symbols \(w_1,\hdots, w_n \in K_2(L;E)\) with \(\Delta(w_1), \hdots, \Delta(w_n)\) all \(\mathbb{Z}/p\)-linearly independent, where \(m\) is the number of places of \(F\) above \(v\).
\end{corollary}

\section{Locally non-trivial symbols}\label{sec:non-trivial_symbols}

Let \(K = \mathbb{Q}(\sqrt{-D})\) an imaginary quadratic field of class number 1; recall that there are finitely many such fields, which are given by a choice of \(D \in \{1,2,3, 7,11, 19, 43, 67, 163\}\). Let \(E/K\) an elliptic curve with CM by \(\mathcal{O}_K\), let \(p\geq 5\) a prime which splits in \(K/\mathbb{Q}\), and let \(v\mid p\) a place of \(K\). Recall that \(E\) has good ordinary reduction at \(v\) since \(p\) splits in \(K/\mathbb{Q}\). Let \(F/K\) a field extension in which \(v\) splits completely, let \(\pi\in \mathcal{O}_K\) an irreducible corresponding to \(v\), and set \(L = F(E[\pi])\). Let \(A\in E[\pi]\) non-zero and let $w\mid v$ a place of $L$. The goal of this section is to make more explicit the non-triviality criterion presented in \autoref{thm:nontriviality_criterion} as applied to symbols of the form \(\{A_w,P\}_{L_w/L_w} \in K_2(L_w;E)/p\), where \(P\) is chosen from \(E_u(F_u)\).

By \autoref{thm:nontriviality_criterion}, the possibility of non-zero local symbols only arises when one has \(|\bar{E_u}(\mathbb{F}_u)|\) divisible by $p$. If we further assume that \(v\) splits completely in \(F/K\), then we have \(F_u = K_v = \mathbb{Q}_p\), and so \(|\bar{E_u}(\mathbb{F}_u)| = |\bar{E_v}(\mathbb{F}_v)|\) and \(L_w / \mathbb{Q}_p\) is totally ramified of degree \(p-1\). In this case, we have the following:

\begin{proposition}\label{prop:nontriviality_via_valuations}
Suppose that \(v\) splits completely in \(F/K\) and that \(p \mid |\bar{E_v}(\mathbb{F}_v)|\). For \(P\in E_u(F_u)\),
\begin{align*}
    \curly{A_w,\res_{L_w/F_u}P}_{L_w/L_w} \neq 0 \pmod{p} \quad \Longleftrightarrow \quad u(\widehat{P}) = 1.
\end{align*}
\end{proposition}
\begin{pf}
By \autoref{lem:ses_for_p_(co)torsion}, our assumption that $p \mid |\bar{E_v}(\mathbb{F}_v)|$ together with the fact that $E[\pi]\subseteq E(L)$ by construction imply that $E_w[p] \subseteq E_w(L_w)$. Thus \autoref{lem:nontriviality_via_valuations} applies, so the given symbol is non-trivial modulo \(p\) if and only if \(w(\widehat{A_w}) + w(\res_{L_w/F_u}\widehat{P}) = p\) (note that the map \(P\mapsto \widehat{P}\) commutes with restriction maps by \autoref{rmk:splitting_of_Mackey_cotorsion}). Since \(A_w \in \widehat{E}(L_w)[p]\) already, we have that \(\widehat{A_w} = A_w\). Since \cite[Theorem IV.6.1]{Silverman_2009} gives that
\begin{align*}
    1 \leq w(A_w) \leq \frac{e(L_w/\mathbb{Q}_p)}{p-1} = 1,
\end{align*}
\(w(A_w) = 1\) and it suffices to show that \(w(\res_{L_w/F_u}\widehat{P}) = p-1\) exactly when \(u(\widehat{P}) = 1\).

To do this, note that for each \(i\) the restriction map satisfies
\begin{align*}
    \res_{L_w/F_u}\big(\mathcal{D}_{u}^i\big) \subseteq \mathcal{D}_w^{i\cdot(p-1)}.
\end{align*}
Since composition with the norm going the other way is just multiplication by \(\bracket{L_w :F_u} = p-1\), this restriction map is injective. Thus, for \(P\in E_u(F_u)\) we have that \(\widehat{P}\in \mathcal{D}_{u}^i\) if and only if \(\res_{L_w/F_u}(\widehat{P})\in \mathcal{D}_{L_w}^{i(p-1)}\), and applying this for \(i=1\) we're done.
\end{pf}

\begin{remark}
    This non-triviality criterion is equivalent to that in \cite[Theorem 4.7]{Gazaki_Koutsianas_2024} when \(P\) is defined over \(\mathbb{Q}\).
\end{remark}

It is reasonable to wonder at this point how often these assumptions can actually be met, since attaining \(p \mid |\bar{E_u}(\mathbb{F}_u)|\) for a fixed elliptic curve may require going up to an extension \(F/K\) which has inertial degree \(p-1\) at \(v\) (namely, the extension \(F = K(E[\bar{\pi}])/K\)). We say that a tuple \((E, v\mid p)\) is \textit{admissible} for \(K\) if \(p\geq 5\) is a prime which splits in \(K/\mathbb{Q}\) and \(E\) is an elliptic curve over \(K\) with complex multiplication by \(\mathcal{O}_K\) such that \(p \mid |\bar{E_v}(\mathbb{F}_v)|\). Using this language, we wish to know for which \(K\) and which values of \(p\) an admissible tuple can be found.

Letting \(\bar{v}\) the other place of \(K\) lying above \(p\), we note that \((E, v\mid p)\) is admissible if and only if \((\prescript{\sigma}{}{E}, \bar{v}\mid p)\) is admissible, where \(\prescript{\sigma}{}{E}\) denotes the elliptic curve obtained by applying the non-trivial automorphism of \(K/\mathbb{Q}\) to the coefficients of \(E\) \cite[Theorem II.2.2]{Silverman_1994}. In particular, existence of admissible tuples does not depend on the choice of \(v \mid p\).

\begin{lemma}\label{lem:all_cases_for_nontriv_via_vals_propn}~
    \begin{enumerate}
        \item If \(D = 1\), then all admissible tuples for \(K\) have \(p=5\) and \(|\bar{E_v}(\mathbb{F}_v)| = 10\).
        \item If \(D \in \{2, 7\}\), then there are no admissible tuples for \(K\).
        \item Otherwise, there are at least 2 (and possibly infinitely many) primes \(p\) for which there exist admissible tuples \((E,p)\) for \(K\), and all have \(p = |\bar{E_v}(\mathbb{F}_v)|\).
    \end{enumerate}
\end{lemma}
\begin{proof}
    Let \(E:y^2 = f_a(x)\) denote the family of elliptic curves with complex multiplication by \(\mathcal{O}_K\) given in \cite[Tableau 1]{Joux_Morain_1995}. Note that without loss of generality, we may assume that \(a\in \mathcal{O}_K\). Also, since since \(p\) splits in \(K/\mathbb{Q}\) by assumption, \(\mathbb{F}_v = \mathbb{F}_p\), and the order of the special fiber \(|\bar{E_v}(\mathbb{F}_v)|\) depends only on the residue of \(a\) modulo \(v\).

    We first consider the case where \(|\bar{E_v}(\mathbb{F}_v)| = np\) for some \(n\geq 2\). By the Hasse bound, 
    \[
        (n-1)p - 1 = \big| |\bar{E_v}(\mathbb{F}_v)| - (p+1) \big| \leq 2\sqrt{p}.
    \]
    The only way this inequality can hold for primes \(p\geq 5\) is if \(n = 2\) and \(p=5\), and an exhaustive search using a Sage script of the five possible residues of \(a\) modulo \(v\) for each family shows that the only occurance of this is when \(D=1\) and \(a\equiv 3 \pmod{v}\).

    We now wish to show that for \(D \in \{1,2,7\}\), there are no \(p\geq 5\) and \(E/K\) as above for which \(p = |\bar{E_v}(\mathbb{F}_v)|\). By \cite{Deuring_1941}, we have that \(|\bar{E_v}(\mathbb{F}_v)| = p + 1 - (\pi + \bar{\pi})\), where \(\pi\bar{\pi}\) is some factorization of \(p\) in \(\mathcal{O}_K\) (specifically, that for which the endomorphism \([\pi]\) restricts to the Frobenius endomorphism modulo \(v\)).
    
    For \(D=1,2\), one has that \(\mathcal{O}_K = \mathbb{Z}[\sqrt{-D}]\). Writing \(\pi = b + c \sqrt{-D}\) we see that \(\pi + \bar{\pi} = 2b\) is always even, and thus the equality \(|\bar{E_v}(\mathbb{F}_v)| = p\) can never occur.

    For all other \(D\), we have that \(\mathcal{O}_K = \mathbb{Z}[\frac{1 + \sqrt{-D}}{2}]\). Writing \(\pi = \frac{b}{2} + \frac{c}{2}\sqrt{-D}\), we see that \(\pi + \bar{\pi} = b\). Assuming that \(p\) is such that \(|\bar{E_v}(\mathbb{F}_v)| = p\), this gives \(b = 1\), and the equality \(\pi\bar{\pi} = p\) implies that \(c\) is an integral solution to \(4p = 1+ Dc^2\). 
    
    Suppose that there is a solution to this equation when \(D = 7\). Then \(1+7c^2\) is either \(0\) or \(4\) modulo \(8\), but in fact the latter cannot occur as \(1+7c^2 \equiv 4 \pmod{8}\) implies that \(c^2 \equiv 5\pmod{8}\), which has no solution. Thus, \(2\mid p\), a contradiction. 

    To show the final assertion, we provide for each \(D\) a complete list of the primes \({5\leq p < 1000}\) for which there exist admissible tuples \((E,p)\), obtained via checking each possible residue of \(a\) with a script in the accompanying \href{https://github.com/mwills758/locally_non-trivial_cycles/}{Sage notebook}:
    \[
        \begin{tabular}[b]{c|l}
            \(D\) & \(p\)\\ \hline
            3 & 7, 19, 37, 61, 127, 271, 331, 397, 547, 631, 919\\
            11 & 223, 619\\
            19 & 5, 43, 233\\
            43 & 11, 97, 269\\
            67 & 17, 151, 419, 821\\
            163 & 41, 367.
        \end{tabular}\qedhere
    \]
\end{proof}

\begin{remark}
    For \(D = 3\), it is noted in \cite[Example 2.4]{Gazaki_Koutsianas_2024} that any prime \(p\) of the form \(4p = 1+3c^2\) can form an admissible tuple \((E,p)\). These are the \textit{Cuban primes}, which are also characterized as being the difference of consecutive cubes, and appear as sequence A002407 in the OEIS. This sequence is conjectured to be infinite. The other sequences of primes in the above table do not presently appear in the OEIS.
\end{remark}

\begin{remark}\label{rmk:minimal_models_for_admissible_tuples}
    Fix \(K = \mathbb{Q}(\sqrt{-D})\) as above, and let \(y^2 = f_a(x)\) the corresponding equation from \cite[Tableau 1]{Joux_Morain_1995}. From the computations at the end of the proof of \autoref{lem:all_cases_for_nontriv_via_vals_propn}, we see that none of the primes dividing the coefficients of \(f_a(x)\) form an admissible tuple for \(K\). Consequently, if \((E,v\mid p)\) is admissible for \(K\), then \(E\) has a minimal model over \(\mathcal{O}_K\) which is in short Weierstrass form and has coefficients of \(v\)-adic valuation zero.
\end{remark}

\subsection{The \texorpdfstring{\(\bar{\pi}\)}{pi\_bar}-torsion of \texorpdfstring{\(E\)}{E}}
Let \((E, v\mid p)\) be admissible for \(K = \mathbb{Q}(\sqrt{-D})\), and let \(\pi\), \(F\), \(L\), \(A\), \(w\), and \(u\) be as at the beginning of \autoref{sec:non-trivial_symbols}. By \autoref{lem:all_cases_for_nontriv_via_vals_propn}, we may write \(|\bar{E_v}(\mathbb{F}_v)| = dp\) where \(d\) is 2 if \(D= 1\) and \(p = 5\), and \(d= 1\) otherwise. For \(P \in E_u(F_u)\), non-triviality of \(\{A_w,\res P\}_{L_w/L_w}\) modulo \(p\) is equivalent to having \(u(\widehat{P}) = 1\) by \autoref{prop:nontriviality_via_valuations}, and by \autoref{lem:P_hat_computation} computing this valuation requires understanding the \(x\)-coordinates of points in \(E_u[\bar{\pi}]\). The goal of this section is to characterize such \(x\)-coordinates.

Our main tool for doing so will be analyzing the kernel polynomial associated to \([\bar{\pi}]\in \End_K(E)\). This is a polynomial \(\phi_{\bar{\pi}} \in K[x]\) of degree \(\frac{p-1}{2}\), the \(\bar{K}\)-roots of which are exactly those values which appear as \(x\)-coordinates of points in \(E[\bar{\pi}]\). That \(E_u[\bar{\pi}]\) is \(F_u\)-rational means that \(\phi_{\bar{\pi}}\) splits into linear factors over \(F_u = \mathbb{Q}_p\), and \autoref{lem:P_hat_computation} gives that these roots are all distinct modulo \(u\).

It will be useful to simultaneously compute the \(\bar{\pi}\)-torsion of all elliptic curves with the same CM type. To that end, we see by \cite[Tableau 1]{Joux_Morain_1995} that all elliptic curves with complex multiplication by \(\mathcal{O}_K\) are parametrized in a 1-dimensional family by \(E_a: y^2 = f_a(x)\), where \(f_a(x)\) has (or can be put into) the form 
\begin{align*}
    f_a(x) = \begin{cases}
        x^3 + ax & D = 1\\
        x^3 + a & D = 3\\
        x^3 + n_Da^2 x + m_Da^3 & \text{otherwise}
    \end{cases}
\end{align*}
for some integers \(n_D,m_D\) depending on \(D\), and $a$ ranges over $K^\times$. We may regard this as an elliptic curve \(\mathcal{E}\) over \(\mathbb{G}_m\). Further, $\mathcal{E}$ has complex multiplication by $\mathcal{O}_K$, and we let $\Phi_{\bar{\pi}}(x)$ denote the kernel polynomial of \([\bar{\pi}] \in \End_{\mathbb{G}_m}(\mathcal{E})\) (recall that $p\neq 2$, so there is no dependence on $y$). This polynomial has coefficients in $K[a]$, and regarded as an element $\Phi_{\bar{\pi}}(x,a) \in K[x,a]$ has $x$-degree $\frac{p-1}{2}$. The complex multiplication on \(\mathcal{E}\) specializes to that on each of the fibers, from which it follows that if $E$ is the fiber of $\mathcal{E}$ lying above some $a_0\in K^\times$, then the kernel polynomial $\phi_{\bar{\pi}}(x)$ of $[\bar{\pi}]$ as an endomorphism of $E$ is given by $\phi_{\bar{\pi}}(x) = \Phi_{\bar{\pi}}(x,a_0)$. 
We similarly let \(\Phi_p\) and \(\Phi_\pi\) denote the kernel polynomials of \([p]\) and \([\pi]\) on \(\mathcal{E}\).

\begin{lemma}\label{lem:pi_bar_kernel_poly_facts}
    \begin{enumerate}
        \item \(\Phi_{\bar{\pi}}\) has coefficients in \(\mathcal{O}_K\) and is homogeneous of degree \(\frac{p-1}{2}\) after weighting \(a\) by
        \begin{align*}
            w = \frac{|\mathcal{O}_K^\times|}{2} = \begin{cases}2 & D =1 \\ 3 & D = 3\\ 1 & \text{otherwise.}\end{cases}
        \end{align*}
        \item \(\Phi_{p}\) factors into irreducibles in \(K[x,a]\) (and thus in \(\mathcal{O}_K[x,a]\)) as \(\Phi_p = \Phi_{\pi}\cdot \Phi_{\bar{\pi}}\cdot g\) for some \(g\in K[x,a]\).
    \end{enumerate}
\end{lemma}
\begin{proof}
    \begin{enumerate}
        \item First, note that every fiber \(\mathcal{E}_a\) with \(a\in K^\times\) is isomorphic over a suitable extension of \(K\) to the fiber \(\mathcal{E}_1\) via the isomorphism \((x,y) \mapsto (a^{-1/w}x, a^{-3/2w}y)\). This isomorphism preserves the complex multiplication, and it follows that $\Phi_{\bar{\pi}}(x,a)$ and $\Phi_{\bar{\pi}}(a^{-1/w}x, 1)$ have the same zeros. We have already observed that $\Phi_{\bar{\pi}}$ is polynomial in $x$ and $a$, so it must be that $\Phi_{\bar{\pi}}(x,a) = a^n\Phi_{\bar{\pi}}(a^{-1/w}x, 1)$ for some $n$ large enough to clear denominators. In fact, this $n$ must be as small as possible; since $\Phi_{\bar{\pi}}$ divides $\Phi_p$, the leading $x$-coefficients also satisfy this divisibility relation, and the leading coefficient of $x$ in $\Phi_p$ is simply $p$ so has no dependence on $a$. Thus, $\Phi_{\bar{\pi}}(x,a)$ is homogeneous in $x$ and $a^{1/w}$, so has the weighted homogeneity required.
        
        \item We first note that $\Phi_\pi$ and $\Phi_{\bar{\pi}}$ both divide $\Phi_p$, as $\pi$ and $\bar{\pi}$ divide $p$ in $\mathcal{O}_K = \End_{\mathbb{G}_m}(\mathcal{E})$. It remains to show that both of these are irreducible, share no common factors, and that the remaining factor of $\Phi_p$ is also irreducible. It suffices to show all of these things for a fixed $a_0 \in K^\times$. Let $\phi_{\pi}(x) = \Phi_{\pi}(x,a_0)$, and similarly define $\phi_{\bar{\pi}}(x)$ and $\phi_p(x)$; these are the kernel polynomials for the corresponding endomorphisms of the elliptic curve $E = \mathcal{E}_{a_0}$ defined over $K$.
        
        First, note that \(\phi_\pi\) is irreducible, since the extension \(K(E[\pi])/K\) is of degree \(p-1\) by \cite[Corollary 5.20.ii]{Rubin_1999}. Indeed, if there were a smaller irreducible factor, then taking a root of it together with the corresponding \(y\) coordinate would give a non-zero point \(P \in E[\pi]\) defined over an extension of degree less than \(p-1\), but \(E[\pi]\) is cyclic and this \(P\) would generate all of \(E[\pi]\), a contradiction. An identical argument holds for \(\phi_{\bar{\pi}}\). Also note that $\phi_{\pi}$ and $\phi_{\bar{\pi}}$ share no roots over $\bar{K}$, as the subgroups $E[\pi]$ and $E[\bar{\pi}]$ have trivial intersection.
        
        Now we wish to show that there are no non-trivial factors of \(\phi_p/(\phi_\pi \cdot \phi_{\bar{\pi}})\). Suppose that there is such a factor \(h\), fix one root \(x_0\) of \(h\), and let \(L'/K\) be the extension obtained by adjoining \(x_0\) as well as some \(y_0\) such that \(P = (x_0,y_0)\) is a point in \(E[p]\). This extension then has degree less than \((p-1)^2 = [K(E[p]):K]\); in particular, \(E[p](L')\) is a one-dimensional subspace of \(E[p]\). Furthermore, \(E[p](L')\) must intersect both \(E[\pi]\) and \(E[\bar{\pi}]\) trivially, as \(L'\) was formed by adjoining some \(P \in E[p] \setminus (E[\pi] \cup E[\bar{\pi}])\). 
        
        We know that \(p=\pi\cdot \bar{\pi}\) splits in \(K/\mathbb{Q}\). Further, we have by \autoref{lem:places_of_L_above_p_structure} that \(\pi\) is totally ramified in \(K(E[\pi])/K\). The same Lemma gives that \(\bar{\pi}\) is unramified in \(K(E[\pi])/K\), and in fact uniqueness in \autoref{lem:P_hat_computation} (together with an argument similar to that in the proof of \autoref{lem:when_v_splits_in_naive_quadratic} below) gives that \(\bar{\pi}\) splits completely in this extension. Thus, in \(K(E[p]) = K(E[\pi], E[\bar{\pi}])/K\), the places corresponding to \(\pi\) and \(\bar{\pi}\) each factor into \((p-1)^{\text{st}}\) powers of \(p-1\) distinct prime ideals.
        
        Now, consider the splitting behavior of \(L'/K\) at \(\pi\) and \(\bar{\pi}\). Since \([L':K] < (p-1)^2\), it cannot have the same behavior as described for \(K(E[\pi])/K\) above. In particular, we have one of the following:
        \begin{enumerate}
            \item \(\pi\) has less than \(p-1\) distinct places lying above it,
            \item the places above \(\bar{\pi}\) are ramified of degree less than \(p-1\),
            \item \(\bar{\pi}\) has less than \(p-1\) distinct places lying above it, or
            \item the places above \(\pi\) are ramified of degree less than \(p-1\).
        \end{enumerate}
        Without loss of generality, assume that either (a) or (b) holds, and consider the extension \(L'(E[\pi])/K\). Then whichever of (a) and (b) held before still holds for this new extension, so it cannot be that \(L'(E[\pi]) = K(E[p])\). However, \(E[p](L(E[\pi]))\) now contains two linearly independent points, so must be all of \(E[p]\), a contradiction. \qedhere
    \end{enumerate}
\end{proof}

If \(a\) is such that \((\mathcal{E}_a, v\mid p)\) is admissible for \(K\), then \autoref{lem:pi_bar_kernel_poly_facts} allows for the computation of the kernel polynomial \(\phi_{\bar{\pi}}\) associated to an arbitrary fiber \(\mathcal{E}_a\) as follows. First, compute the \(p^{\text{th}}\) division polynomial of \(E = E_1\) as \(\phi_p \in K[x]\). Factoring \(\phi_p\) over \(K\) into three irreducibles as above, two of degree \(\frac{p-1}{2}\) and one of degree \(\frac{(p-1)^2}{2}\). The two factors of degree \(\frac{p-1}{2}\) then correspond to \(\phi_{\bar{\pi}}\) and \(\phi_{\pi}\). Homogenizing with the appropriate weight, we get the specializations of \(\Phi_{\bar{\pi}}\) and \(\Phi_{\pi}\) to \(\mathcal{E}_a\), and these are distinguished by the presence and absence respectively of roots modulo \(\pi\).

This algorithm has been implemented in the accompanying \href{https://github.com/mwills758/locally_non-trivial_cycles/}{Sage notebook}. The computational bottleneck lies in factoring the polynomial \(\phi_p(x)\), as this polynomial has degree \(O(p^2)\); this step already takes a couple minutes for \(p\approx 100\) on a laptop. If one wanted to speed this algorithm up to access higher values of \(p\), one approach would be to explicitly describe the CM action, then use an implementation of point addition in Jacobian coordinates to directly compute the endomorphism \(\phi_{\bar{\pi}}\) as applied to a generic point \(P = (x,y)\). The main issue with this is that explicit formulas for CM beyond \(D = 7\) can get rather messy, and do not seem to be widely available. 

Once we have computed \(\phi_{\bar{\pi}}(x,a)\), we may use it to characterize the \(x\)-coordinates of the \(\bar{\pi}\)-torsion points of \(E_u\) with a simple Taylor expansion:

\begin{lemma}\label{lem:pi_bar_torsion_coeff_relation}~
    Let \(\phi = \Phi_{\bar{\pi}}\). Suppose that \(\phi(x,a)=0\) for some \(x,a\in \mathcal{O}_u\), and write \(x = x_0 + x_1p + O(p^2)\) and \(a = a_0 + a_1p + O(p^2)\). Then \(x_1\) is uniquely characterized modulo \(p\) as satisfying the equation
    \[
        \frac{\phi(x_0,a_0)}{p} + x_1 \frac{\partial \phi}{\partial x}(x_0,a_0) + a_1 \frac{\partial \phi}{\partial a}(x_0,a_0) \equiv 0 \pmod{p}.
    \]
\end{lemma}
\begin{proof}
    Consider the Taylor expansion of \(\phi\) in powers of \(p\) around the point \((x_0,a_0)\), given by
    \begin{align*}
        \phi(s,t) = \phi(x_0,a_0) + \left((s-x_0) \frac{\partial \phi}{\partial s}(x_0,a_0) + (t-a_0) \frac{\partial \phi}{\partial t}(x_0,a_0)\right)p + O(p^2).
    \end{align*}
    Reducing the equation \(\phi(x,a) = 0\) modulo \(p\), we get that \(\phi(x_0,a_0)\) must be divisible by \(p\), and substituting \((x,a)\) into the above equation gives that 
    \begin{align*}
        \left(\frac{\phi(x_0,a_0)}{p} + x_1\frac{\partial \phi}{\partial s}(x_0,a_0) + a_1 \frac{\partial \phi}{\partial t}(x_0,a_0)\right)p \equiv 0 \pmod{p^2},
    \end{align*}
    which is equivalent to the desired condition. That \(x_1\) is uniquely characterized by this equation is guaranteed by \(\frac{\partial \phi}{\partial s}(x_0,a_0)\) being non-zero modulo \(p\), which holds by distinctness of the \(\bar{\pi}\)-torsion \(x\)-coordinates modulo \(v\).
\end{proof}

\subsection{Local non-triviality criteria}

Once again, let \((E, v\mid p)\) be admissible for \(K = \mathbb{Q}(\sqrt{-D})\), and let \(\pi\), \(F\), \(L\), \(A\), \(w\), and \(u\) be as at the beginning of \autoref{sec:non-trivial_symbols}. Write \(E = \mathcal{E}_a\) for some \(a \in \mathcal{O}_K\), where \(\mathcal{E}:y^2 = f_a(x)\) is family of elliptic curves with complex multiplication by \(\mathcal{O}_K\) as given in \cite[Tableau 1]{Joux_Morain_1995}. Let \(\phi\) denote the kernel polynomial of \([\bar{\pi}] \in \End_K(E)\).

We may now use the characterizations of \(E[\bar{\pi}]\) developed in the previous section to give our first main Theorems, which are concrete non-triviality criteria for local symbols of the form discussed at the start of \autoref{sec:non-trivial_symbols}. We will give two such criteria, one for the generic case where \(p = |\bar{E_v}(\mathbb{F}_v)|\) and another for the single case where we have \(p \mid |\bar{E_v}(\mathbb{F}_v)|\) without equality.

\begin{theorem}\label{thm:no_inertia_nontriviality}
    Suppose that \(p = |\bar{E_v}(\mathbb{F}_v)|\), and write \(a = a_0 + a_1p + O(p^2)\) in \(\mathcal{O}_v\). Let \(P \in E_u(F_u)\). 
    \begin{enumerate}
        \item If \(u(x(P)) < 0\), then \(\{A_w,P_w\}_{L_w/L_w} \neq 0 \pmod{p}\) if and only if \(u(x(P))=-2\).
        \item If \(u(x(P)) \geq 0\), write \(x(P) = b_0 + b_1p + O(p^2) \in \mathcal{O}_u\). Then \(\{A_w,P_w\}_{L_w/L_w} \neq 0 \pmod{p}\) if and only if
        \begin{align}\label{eqn:Taylor_expansion_criterion_for_nontriviality}
            \frac{\phi(b_0, a_0)}{p} + b_1 \frac{\partial \phi}{\partial x}(b_0,a_0) + a_1 \frac{\partial \phi}{\partial a}(b_0,a_0) \not\equiv 0 \pmod{p}.
        \end{align}
    \end{enumerate}
\end{theorem}
\begin{proof}
    By \autoref{prop:nontriviality_via_valuations}, \(\{A_w, P_w\}_{L_w/L_w} \neq 0 \pmod{p}\) if and only if \(u(\widehat{P_u}) = 1\). If \(u(x(P)) < 0\), then \(P = \widehat{P}\) and \(u(\widehat{P}) = \frac{u(x(P))}{-2}\). Otherwise, \autoref{lem:P_hat_computation} provides us a unique non-zero \(T \in E[\bar{\pi}]\) such that \(u(\widehat{P}) = u(x(P) - x(T)) \geq 1\). We may then write \(x(T) = b_0 + x_1p + O(p^2)\), and \autoref{lem:pi_bar_torsion_coeff_relation} gives that \(u(x(P) - x(T)) > 1\) if and only if Equation (\ref{eqn:Taylor_expansion_criterion_for_nontriviality}) holds, as desired.
\end{proof}

The accompanying \href{https://github.com/mwills758/locally_non-trivial_cycles/}{Sage notebook} computes the $p$-adic expansions to order $O(p^2)$ of the $x$-coordinates appearing in $E[\bar{\pi}]$ in terms of $a$; that is, the values of $b_0$ and $b_1$ which satisfy Equation (\ref{eqn:Taylor_expansion_criterion_for_nontriviality}). This allows one to easily identify whether a point $P$ will give rise to a global zero-cycle which is locally non-trivial modulo $p$.

A similar result holds for the case where \(D = 1\) and \(p=5\). 
Some additional complications are introduced by the presence of additional torsion points, but we are able to give a concrete criterion as follows.

\begin{theorem}\label{thm:D_equals_1_nontriviality_criterion}
    Suppose that \(p \mid |\bar{E_v}(\mathbb{F}_v)|\) without equality, so \(K = \mathbb{Q}(i)\), \(p=5\), \(E\) is given by \(y^2 = x^3 + ax\), and \(|\bar{E_v}(\mathbb{F}_v)| = 10\). Let \(P \in E_u(F_u)\).
    \begin{enumerate}
        \item If \(u(x(P)) <0\), then \(\{A_w,P_w\}_{L_w/L_w} \neq 0 \pmod{p}\) if and only if \(u(x(P)) = -2\).
        \item If \(u(x(P)) = 0\), write \(a = 3 + a_1p + O(p^2)\) and \(x(P) = b_0 + b_1p + O(p^2)\) in \(\mathcal{O}_u\) with \(b_0 \in \{1,\hdots, 4\}\). Then \(\{A_w,P_w\}_{L_w/L_w} \neq 0 \pmod{p}\) if and only if \(b_1 \not\equiv b_0a_1 + \varepsilon(b_0) \pmod{5}\), where \(\varepsilon\) is defined by
        \begin{align*}
            \begin{array}{c|cccc}
                b_0 & 1 & 2 & 3 & 4 \\ \hline
                \varepsilon(b_0) & 3 & 4 & 3 & 1. \\
            \end{array}
        \end{align*}    
        \item If \(u(x(P)) > 0\), then \(\{A_w,P_w\}_{L_w/L_w} \neq 0 \pmod{p}\) if and only if \(u(x(P)) = 2\).
    \end{enumerate}
\end{theorem}
\begin{proof}
    The case of \(u(x(P)) < 0\) is identical to that in the proof of \autoref{thm:no_inertia_nontriviality}.

    If \(u(x(P)) > 0\), then \(P\) restricts to the 2-torsion point \((0,0) \in \bar{E_u}(\mathbb{F}_u)\), so \(2P = \widehat{2P} \in \widehat{E_u}(\mathfrak{m}_u)\); note that \(u(\widehat{2P}) = u(\widehat{P})\) by \autoref{rmk:point_valuations_same_after_coprime_to_p_scaling}. Now, \(a\equiv 3 \pmod{p}\) implies that \(2u(y(P)) = u(x(P))\) by taking valuations in the equation defining \(E\), and doing the same to the point-doubling formula gives that \(u(x(2P))= -u(x(P))\), from which the claim follows.

    Now assume that \(u(x(P)) = 0\). Without loss of generality, assume that \(\pi\) is of the form \(\pi = 2 \pm i\). Using Sage, we compute \(\phi = \phi_{\bar{\pi}} = \bar{\pi}x^2 \mp ai\). Plugging this into \autoref{lem:pi_bar_torsion_coeff_relation} and simplifying, we see that if a point \(Q\) is such that \(\widehat{Q} = Q +T\) for some non-zero \(T\in E[\bar{\pi}]\), then \(\{A_w,Q_w\}_{L_w/L_w} = 0 \pmod{p}\) if and only if
    \begin{align*}
        \frac{\bar{\pi}c_0^2 \mp 3i}{p} - 2c_1c_0 + 2a_1 \equiv 0 \pmod{p},
    \end{align*}
    where \(x(Q) = c_0 + c_1p + O(p^2)\). (Note that the ambiguity in sign in the first term is resolved by the relation \(i \equiv \mp 2 \pmod{p}\) in \(\mathcal{O}_u\).) Since the roots of \(\phi\) in \(F_u\) are equivalent to either 1 or 4 modulo \(p\), we see that if \(b_0 = 1,4\) we may apply this criterion directly, solving for \(b_1\) as
    \begin{align*}
        b_1 \equiv b_0a_1 + \frac{b_0(\bar{\pi}b_0^2 \mp 3i)}{2p} \pmod{p}
    \end{align*}
    (note that both of these values of \(b_0\) are their own inverses modulo \(p\)).
    
    In the case that \(b_0 = 2,3\), \autoref{lem:P_hat_computation} provides that such a \(T\) exists for \(2P\), and we compute
    \begin{align*}
        x(2P) = \left(\frac{3x(P)^2 + a}{2y(P)}\right)^2 - 2x(P) \equiv -2x(P) \pmod{p};
    \end{align*}
    since \(a\equiv 3 \pmod{p}\), the numerator of the first term vanishes modulo \(p\). Now, the above equivalence simplifies to
    \begin{align*}
        b_1 \equiv b_0a_1 + \frac{b_0(4\bar{\pi}b_0^2 \mp 3i)}{4p} \pmod{p},
    \end{align*}
    and evaluating the final term of this and our previous equivalence at \(b_0\) defines the given function \(\varepsilon\).
\end{proof}

\begin{remark}\label{rmk:p_minus_1_on_p_proportion}
    Fix \(p\) and \(K\), and suppose we have a family of tuples \((E_t, v_t \mid p)\) admissible for \(K\) with associated fields \(F_t/K\) in which \(v_t\) splits completely and places \(u_t \mid v_t\) of \(F_t\). It follows from these Theorems that if points among all \(E_t(F_t)\) are sampled randomly in such a way that their \(x\)-coordinates are uniformly distributed modulo \(p^2\) in \(F_{u_t}\) across the possible residues, one should expect that \(\frac{p-1}{p}\) of these points may be used to construct locally non-trivial symbols modulo \(p\) as above. 
    
    While proving this uniformity in any particular case seems daunting, it does offer an explanation for the data collected in \cite[Theorem A.2]{Gazaki_Koutsianas_2024}. There, the authors fixed \(p = 7\) and \(K = \mathbb{Q}(\sqrt{-3})\), considered the family of curves \(E_t : y^2 = x^3 + (-2 + 7t)\) for \(t\in \mathbb{Z}\), fixed uniform choices of \(\pi = \frac{1+3\sqrt{-3}}{2}\) and \(F_t = K\), and sampled all linearly independent points of infinite order from among the \(E_t(\mathbb{Q})\) with \(|t| < 5000\). They found that 86.68\% of these points would give rise to non-trivial local symbols as in \autoref{thm:no_inertia_nontriviality}, which is very near to the \(\frac{6}{7}\approx 85.7\%\) expected.
\end{remark}

\subsection{Applications to na\"ive quadratic points}

Let \((E, v\mid p)\) be admissible for \(K = \mathbb{Q}(\sqrt{-D})\), and let \(\pi\) as before. By \autoref{rmk:minimal_models_for_admissible_tuples}, we may take \(E:y^2 = f(x)\) where \(f(x) = x^3 + Ax + B\) with \(v(A), v(B) = 0\). We wish to apply the results of the previous section specifically to the case in which the extension field \(F\) is constructed by adjoining a na\"ive quadratic point of \(E\) to \(K\), which is to say a point of the form \((b,\sqrt{f(b)})\) for some \(b\in K\). The first step is to determine the conditions on \(b\) under which \(v\) splits in \(F = K(\sqrt{f(b)})/K\).

\begin{lemma}\label{lem:when_v_splits_in_naive_quadratic}~
    \begin{enumerate}
        \item If \(v(b)<0\), then \(v\) splits in \(F/K\) if and only if \(v(b) = 2n\) is even and writing \(b = \frac{b'}{p^{2n}}\) one has that \(b'\) reduces to a square in \(\mathbb{F}_v\).
        \item If \(v(b) \geq 0\) and \(p = |\bar{E_v}(\mathbb{F}_v)|\), let \(\alpha\in \mathbb{F}_v\) the reduction of \(b\). Then the following are equivalent:
        \begin{enumerate}
            \item there exist a unique pair of points \(\pm T \in E_v[\bar{\pi}]\) with \(x(T)\) reducing to \(\alpha\) modulo \(v\),
            \item \(f(\alpha)\) is a non-zero square in \(\mathbb{F}_v\), and
            \item  \(v\) splits in \(F/K\).
        \end{enumerate}
        \item If \(v(b) = 0\) and \(p \mid |\bar{E_v}(\mathbb{F}_v)|\) without equality, then \(v\) always splits in \(F/K\).
        \item If \(v(b) > 0\) and \(p \mid |\bar{E_v}(\mathbb{F}_v)|\) without equality, then \(v\) splits in \(F/K\) if and only if \(v(b)=2n\) is even and writing \(b = b'p^{2n}\) one has that \(b'\) reduces to a non-square in \(\mathbb{F}_v\).
    \end{enumerate}
\end{lemma}
\begin{proof}
    \begin{enumerate}
        \item First, suppose that \(v(b) < 0\). Recall that \(v\) splits in \(F/K\) if and only if \(f(b)= b^3 + Ab + B\) is a square in \(K_v\). Evaluating \(v(f(b)) = 3v(b)\) since \(v(b) < 0\), we see that a necessary condition on \(b\) for \(v\) splitting in \(K(\sqrt{f(b)})/K\) is that \(v(b) = 2n\) is even. Writing \(b = \frac{b'}{p^{2n}}\), we note that \[K(\sqrt{f(b)}) = K(\sqrt{(b')^3 + Ab'p^{2n} + Bp^{3n}}).\] Thus, \(v\) splitting in \(K(\sqrt{f(b)})/K\) is equivalent to 
        \begin{align*}
            t^2 - ((b')^3 + Ab'p^{2n} + Bp^{3n})
        \end{align*}
        splitting into distinct linear factors modulo \(v\). Simplifying, we see that this happens if and only if \((b')^3\) is a square modulo \(v\), which occurs exactly when \(b'\) is.

        \item Now consider the case when \(v(b) \geq 0\) and \(p = |\bar{E_v}(\mathbb{F}_v)|\). Note that (a) immediately implies (b), as for \(T \in E_v[\bar{\pi}] \subseteq E_v(K_v)\) we reduce the relation \(f(x(T)) = y(T)^2\) modulo \(v\). Further, (b) implies (c) by Kummer's Theorem on factorization in Dedekind domains; either \(\sqrt{f(b)} \in K\) already, or the minimal polynomial of \(\sqrt{f(b)}\) is \(t^2 - f(b)\) and splits into distinct linear factors modulo \(v\).

        To see that (c) implies (a), let \(u\) be one of the places of \(F:= K(\sqrt{f(b)})\) lying above \(v\). Since \(v\) splits in this extension, we have equality of local fields \(F_u = K_v\), and so \(|\bar{E_u}(\mathbb{F}_u)| = |\bar{E_v}(\mathbb{F}_v)| = p\). We may then take \(d=1\) in \autoref{lem:P_hat_computation} to get a unique pair of points \(\pm T\in E[\bar{\pi}]\) such that \(u(x(P) - x(T)) = v(x(T) - x(P)) \geq 1\), which implies that \(x(T)\) reduces to \(\alpha\) modulo \(v\).

        \item If \(v(b) = 0\) and \(p \mid |\bar{E_v}(\mathbb{F}_v)|\) without equality, by \autoref{lem:all_cases_for_nontriv_via_vals_propn} we have that \(f(x) = x^3 + ax\) for some \(a\equiv 3\pmod{v}\). Checking all non-zero residues, we see that \(v(b) = 0\) implies that \(f(b)\) reduces to a square modulo \(v\), and so \(v\) splits in \(F/K\).
        
        \item Finally, suppose that \(v(b) > 0\) and \(p \mid |\bar{E_v}(\mathbb{F}_v)|\) without equality. As before, we know that \(f(b) = b^3 + ab\), and so a necessary condition for \(v\) to split in \(F/K\) is to have \(v(f(b)) = v(b)\) be even. Writing \(b = b'p^{2n}\), we see that
        \begin{align*}
            F = K(\sqrt{f(b)}) = K(\sqrt{(b')^3p^{4n} + ab'}),
        \end{align*}
        so \(v\) splits in \(F/K\) exactly when \(ab'\) is a square modulo \(v\). Since \(a\) is not a square modulo \(v\), this is equivalent to having \(b'\) also not a square modulo \(v\). \qedhere
    \end{enumerate}
\end{proof}

Now, fix \(b \in K\) such that \(v\) splits in \(F/K\), and let \(P = (b,\sqrt{f(b)}) \in E(F)\). Let \(L = F(E[\pi])\), fix \(A \in E[\pi]\) non-zero, and let \(w\) a place of \(L\) lying above \(v\). We may apply the nontriviality criteria of the previous section to determine when the symbol \(z = \{A_w, P_w\}_{L_w/L_w}\) is non-trivial modulo \(p\).

\begin{corollary}\label{cor:naive_quadratics_nontriviality}~
    \begin{enumerate}
        \item If \(v(b) < -2\), then \(z = 0\) modulo \(p\).
        \item If \(v(b) = -2\), then \(z\) is non-zero modulo \(p\).
    \end{enumerate}
    If the above hypotheses do not hold, then \(v(b) \geq 0\).
    Write \(b = b_0 + b_1p + O(p^2)\) and \(E:y^2 = f_a(x)\) with \(a = a_0 + a_1p + O(p^2)\).
    \begin{enumerate}
        \item[3.] If \(p = |\bar{E_v}(\mathbb{F}_v)|\), then \(z\) is non-zero modulo \(p\) if and only if Equation (\ref{eqn:Taylor_expansion_criterion_for_nontriviality}) from \autoref{thm:nontriviality_criterion} holds.
        \item[4.] If \(p \mid |\bar{E_v}(\mathbb{F}_v)|\) without equality and \(v(b) = 0\), choose \(b_0\) and \(b_1\) such that \(b_0 \in \{1,\hdots, 4\}\). Then \(z\) is non-zero modulo \(p\) if and only if \(b_1\not\equiv b_0a_1 + \varepsilon(b_0) \pmod{5}\), where \(\varepsilon\) is as in \autoref{thm:D_equals_1_nontriviality_criterion}. If \(v(b) = 2\), then \(z\) is non-zero modulo \(p\). Otherwise, \(v(b) > 2\) and \(z = 0\) modulo \(p\).
    \end{enumerate}
\end{corollary}

\section{Infinite families of non-trivial local-to-global principles}\label{sec:linear_independence}

Let \((E, v\mid p)\) be admissible for \(K = \mathbb{Q}(\sqrt{-D})\), and let \(\pi \in \mathcal{O}_K\) an irreducible corresponding to \(v\). The goal of this section is to apply \autoref{cor:naive_quadratics_nontriviality} to construct certain infinite families of extensions \(L/K(E[\pi])\) for which the complex (\ref{eqn:p_primary_F2_complex}) for \(X = (E\times E)_L\) is exact with non-zero middle term. Note that for \(F/K\) of degree 2 in which \(v\) splits completely, \autoref{cor:adelic_F2_structure} tells us that for \(L = F(E[\pi])\) and \(X = (E\times E)_L\), exactness of (\ref{eqn:p_primary_F2_complex}) is shown by finding two global symbols \(w_1,w_2\in K_2(L;E)\), the images of which are \(\mathbb{Z}/p\)-linearly independent in
\begin{align*}
    \limi_n F^2_{\mathbb{A}}(X)/p^n \cong \prod_{w\mid v} K_2(L_w;E_w)/p \cong (\mathbb{Z}/p)^2.
\end{align*}
Our procedure for doing so is as follows. We first look for a point of infinite order \(P\in E(K)\) meeting the criteria of \autoref{thm:no_inertia_nontriviality}; if \(E\) has positive rank over \(K\), then heuristically a point of infinite order should have a high probability (about \(\frac{p-1}{p}\)) of doing so. Once one such point has been found, we may obtain a second by adjoining a na\"ive quadratic point \(Q\) with \(x\)-coordinate meeting the criteria of \autoref{cor:naive_quadratics_nontriviality}. Letting \(F = K(y(Q))\) and \(L = F(E[\pi])\), and fixing a point \(A\in E[\pi]\) we obtain two global symbols
\begin{align*}
    \{A,P\}_{L/L} \quad \text{and} \quad \{A,Q\}_{L/L}
\end{align*}
which are locally non-trivial modulo \(p\) at both places of \(L\) lying above \(v\). As \autoref{cor:naive_quadratics_nontriviality} holds for an open subset of \(\mathcal{O}_K\) (with respect to the topology induced by \(v\)), varying \(x(Q)\) in our construction yields an infinite family of fields \(L\) as desired.

All that remains is to check that the two symbols we construct are in fact \(\mathbb{Z}/p\)-linearly independent.

\begin{proposition}\label{prop:linear_independence_of_symbols}
    Write \(E: y^2 = f(x)\), and let \(F = K(\sqrt{f(b)})\) for some \(b\in \mathcal{O}_K\) satisfying the criteria of \autoref{cor:naive_quadratics_nontriviality}. Assume that \(F\neq K\) and that \(v\) splits in \(F/K\). Let $P\in E(K[\pi])$ such that $P$ pairs non-trivially modulo $p$ with $A$ at some place of $K(E[\pi])$ lying above $v$. Then
    \begin{align*}
        c_P = \Delta(\{A,P\}_{L/L}) \quad \text{and} \quad c_Q = \Delta(\{A,Q\}_{L/L})
    \end{align*}
    are \(\mathbb{Z}/p\)-linearly independent.
\end{proposition}

\begin{pf}
Let \(u_1\) and \(u_2\) the places of \(F\) lying above \(v\), and note that since \(v\) splits completely in \(F/K\) we have equality of local fields \(F_{u_1} = K_v = F_{u_2}\). Since \(F/K\) is Galois with \(\Gal(F/K)\) generated by \(\sigma:\sqrt{f(b)}\mapsto -\sqrt{f(b)}\), the embeddings \(\iota_{u_1},\iota_{u_2}:F\hookrightarrow K_v\) corresponding to the places \(u_1\) and \(u_2\) are related by precomposition by \(\sigma\). 

Now let \(w_1\) and \(w_2\) the places of \(L\) lying over \(u_1\) and \(u_2\) respectively. Letting \(\varpi\) the unique place of \(L' = K(E[\pi])\) over \(v\), an analogous argument to the above shows that \(L_{w_1} = L'_{\varpi} = L_{w_2}\), and again the two embeddings \(\iota_{w_1}, \iota_{w_2}:L\hookrightarrow L'_{\varpi}\) are related by precomposition by the \(L'\)-automorphism of \(L\) given by \(\sigma:\sqrt{f(b)}\mapsto -\sqrt{f(b)}\).

Suppose for contradiction that \(c_P\) and \(c_Q\) are not \(\mathbb{Z}/p\)-linearly independent. Since both vectors are assumed non-zero, this is equivalent to the statement that there is some \(n\in (\mathbb{Z}/p)^\times\) such that \(nc_P + c_Q = 0\). Bilinearity of symbols implies that for both places \(w\mid v\) of \(L\), we have
\begin{align*}
    n\curly{A_w, P_w}_{L_w/L_w} + \curly{A_w, Q_w}_{L_w/L_w} = \curly{A_w, nP_w + Q_w}_{L_w/L_w} \equiv 0
\end{align*}
modulo \(p\). Writing the above symbols in \(L'_{\varpi}\), we note that since \(A\) and \(P\) are both defined over \(L'\) we have that \(A_{w_1} = A_{\varpi} = A_{w_2}\) and \(P_{w_1} = P_{\varpi} = P_{w_2}\), whereas \(Q_{w_1}\) and \(Q_{w_2}\) are related by \(\sigma\). Since \(Q = (b, \sqrt{(f(b))})\) we see that \(Q_{w_2} = -Q_{w_1}\) in \(L'_{\varpi}\). Thus, setting \(w = w_1\) we see that \(n\) satisfies
\begin{align*}
    \curly{A_w, nP_w + Q_w}_{L_w/L_w} \text{ and } \curly{A_w, nP_w - Q_w}_{L_w/L_w}
\end{align*}
both being trivial modulo \(p\). Using \autoref{lem:nontriviality_via_valuations}, we see that both \(\widehat{P_w}\) and \(\widehat{Q_w}\) lie in \(\mathcal{D}^{p-1}_{L_w} \setminus \mathcal{D}^{p}_{L_w}\), and that the above conditions on \(n\) imply that both \(n\widehat{P_w} + \widehat{Q_w}\) and \(n\widehat{P_w} - \widehat{Q_w} \) lie in \(\mathcal{D}^{p}_{L_w}\). But since \(\mathcal{D}^{p-1}_{L_w} / \mathcal{D}^{p}_{L_w} \cong \mathbb{F}_w\) has odd characteristic, taking the difference of the above two elements gives a contradiction \(2\widehat{Q_w} \in \mathcal{D}^{p}_{L_w}\).
\end{pf}

Summarizing, we have shown the following:

\begin{theorem}\label{thm:main_theorem_detailed}
    Let \(K\) a quadratic imaginary field, and let \(E/K\) an elliptic curve with complex multiplication by \(\mathcal{O}_K\). Let \(v\mid p\) a place of \(K\) such that \(p \mid |\bar{E_v}(\mathbb{F}_v)|\). Then whenever \(E(K)\) contains a point \(P\) satisfying the hypotheses of \autoref{thm:no_inertia_nontriviality}, there exist infinitely many quadratic extensions \(L/K(E[\pi])\) for which the complex
    \begin{align}
        \limi_n F^2(X)/p^n \overset{\Delta}{\rightarrow} \limi_n F^2_{\mathbb{A}}(X)/p^n \overset{\varepsilon}{\rightarrow} \Hom(\Br(X), \mathbb{Q}/\mathbb{Z})
    \end{align}
    is exact for \(X = (E\times E)_L\), with generators for \(\Im(\Delta)\) explicitly given as above.    
\end{theorem}
\begin{proof}
    The last thing to be checked is that there are infinitely many \(b\in K\) satisfying the criteria of \autoref{lem:when_v_splits_in_naive_quadratic} and \autoref{cor:naive_quadratics_nontriviality} for which \(F \neq K\); that is, for which \(f(b)\) is not a square in \(K\). This follows since the \(b\in \mathcal{O}_K\) for which \(f(b) \in K^2\) are exactly the \(\mathcal{O}_K\)-integral points of our minimal Weierstrass model of \(E\), and by \cite[Section II.3]{Siegel_1929} (translated in \cite{Fuchs_2014}) this set is finite.
\end{proof}
\newpage
\begin{example}\label{exm:CM_by_D_equals_43}
    The following example is worked out in detail in the accompanying \href{https://github.com/mwills758/locally_non-trivial_cycles/}{Sage notebook}.
    \samepage

    Let \(E\) be the elliptic curve defined over \(\mathbb{Q}\) by \(y^2=f(x)\) with \(f(x) = x^3-3440x+77658\). \(E\) has rank 1 over \(\mathbb{Q}\), with generator \(P = (129/4, 129/8)\). Note that \(E\) has complex multiplication by \(\mathcal{O}_K\) for \(K = \mathbb{Q}(\sqrt{-43})\) after base changing to \(E_K\).
    
    Let \(p = 11\), and let \(\pi\) an irreducible factor of \(p\) in \(\mathcal{O}_K\). Choose any \(b\in\mathcal{O}_K\) such that \(b \equiv 2 \pmod{\pi^2}\), let \(F = K(\sqrt{f(b)})\), and let \(Q = (b,\sqrt{f(b)}) \in E(F)\). Then for \(L = F(E[\pi])\) and \(X = (E\times E)_L\), the complex
    \begin{align*}
        \limi_n F^2(X)/p^n \overset{\Delta}{\rightarrow} \limi_n F^2_{\mathbb{A}}(X)/p^n \overset{\varepsilon}{\rightarrow} \Hom(\Br(X), \mathbb{Q}/\mathbb{Z})
    \end{align*}
    is exact, with \(\Im(\varepsilon) = \{0\}\) and
    \begin{align*}
        \limi_n F^2_{\mathbb{A}}(X)/p^n \cong (\mathbb{Z}/p)^2
    \end{align*}
    generated by
    \begin{align*}
        \Delta(\{A,P\}_{L/L}) \quad \text{and} \quad \Delta(\{A,Q\}_{L/L}),
    \end{align*}
    where \(A\) is a non-zero element of \(E[\pi]\).
\end{example}

\printbibliography

\end{document}